\documentclass{amsart}

\usepackage{amssymb}
\usepackage{amsmath,amsthm}
\usepackage{enumerate}
\usepackage{color}
\usepackage{comment,marginnote}
\usepackage{graphicx}
\usepackage{caption}
\usepackage{subcaption}

\def\sp{\succeq_{\rm SP}}

\newcommand{\C}{\mathcal{C}}

\newcommand{\fsd}{\succeq_{\rm SD}}

\newcommand{\pspace}{\mathcal{X}}
\newcommand{\apspace}{\mathcal{Y}}
\newcommand{\lpr}{\underline{P}}
\newcommand{\upr}{\overline{P}}
\newcommand{\pr}{P}

\newcommand{\ldf}{\underline{F}}
\newcommand{\udf}{\overline{F}}
\newcommand{\df}{F}
\newcommand{\dfs}{\mathcal{F}}
\newcommand{\solp}{\mathcal{M}}
\newcommand{\reals}{\mathbb{R}}
\newcommand{\realsexp}{\overline{\mathbb{R}}}

\newcommand{\partit}{\mathcal{B}}
\newcommand{\domain}{\mathcal K}

\newcommand{\lcop}{\underline{C}}
\newcommand{\ucop}{\overline{C}}
\newcommand{\gambles}{\mathcal{L}}

\newtheorem{theorem}{Theorem}
\newtheorem{lemma}{Lemma}
\newtheorem{proposition}{Proposition}
\newtheorem{corollary}{Corollary}
\theoremstyle{remark}
\newtheorem{remark}{Remark}
\newtheorem{example}{Example}
\newtheorem{definition}{Definition}

\begin{document}


\title{Sklar's theorem in an imprecise setting}

\author{IGNACIO MONTES}

\address{University of Oviedo (Spain)\\
Dept. of Statistics and O.R.\\
Oviedo\\
Spain}
\email{imontes@uniovi.es}

\author{ENRIQUE MIRANDA}

\address{University of Oviedo (Spain)\\
Dept. of Statistics and O.R.\\
Oviedo\\
Spain}
\email{mirandaenrique@uniovi.es}

\author{RENATO PELESSONI}

\address{DEAMS ``B. de Finetti''\\
University of Trieste\\
Piazzale Europa~1\\
I-34127 Trieste\\
Italy}
\email{renato.pelessoni@econ.units.it}

\author{PAOLO VICIG}

\address{DEAMS ``B. de Finetti''\\
University of Trieste\\
Piazzale Europa~1\\
I-34127 Trieste\\
Italy}
\email{paolo.vicig@econ.units.it}

\begin{abstract}
Sklar's theorem is an important tool that connects bidimensional
distribution functions with their marginals by means of a copula.
When there is imprecision about the marginals, we can model
the available information by means of p-boxes, that are pairs of
ordered distribution functions. Similarly, we can consider a set of
copulas instead of a single one. We study the extension of Sklar's
theorem under these conditions, and link the obtained results to
stochastic ordering with imprecision.

\end{abstract}

\maketitle


\section*{Acknowledgement}
*NOTICE: This is the authors' version of a work that was accepted for publication in Fuzzy Sets and Systems. Changes resulting from the publishing process, such as peer review, editing, corrections, structural formatting, and other quality control mechanisms may not be reflected in this document. Changes may have been made to this work since it was submitted for publication. A definitive version was subsequently published in Fuzzy Sets and Systems, vol. 278, 1~November~2015, pages 48-66, doi:10.1016/j.fss.2014.10.007 $\copyright$ Copyright Elsevier http://www.sciencedirect.com/science/article/pii/S0165011414004539

\vspace{0.3cm}
$\copyright$ 2015. This manuscript version is made available under the CC-BY-NC-ND 4.0 license http://creativecommons.org/licenses/by-nc-nd/4.0/

\begin{center}
 \includegraphics[width=2cm]{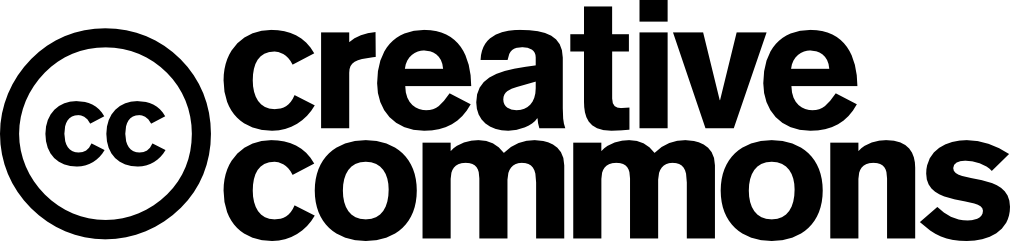}
 \includegraphics[width=2cm]{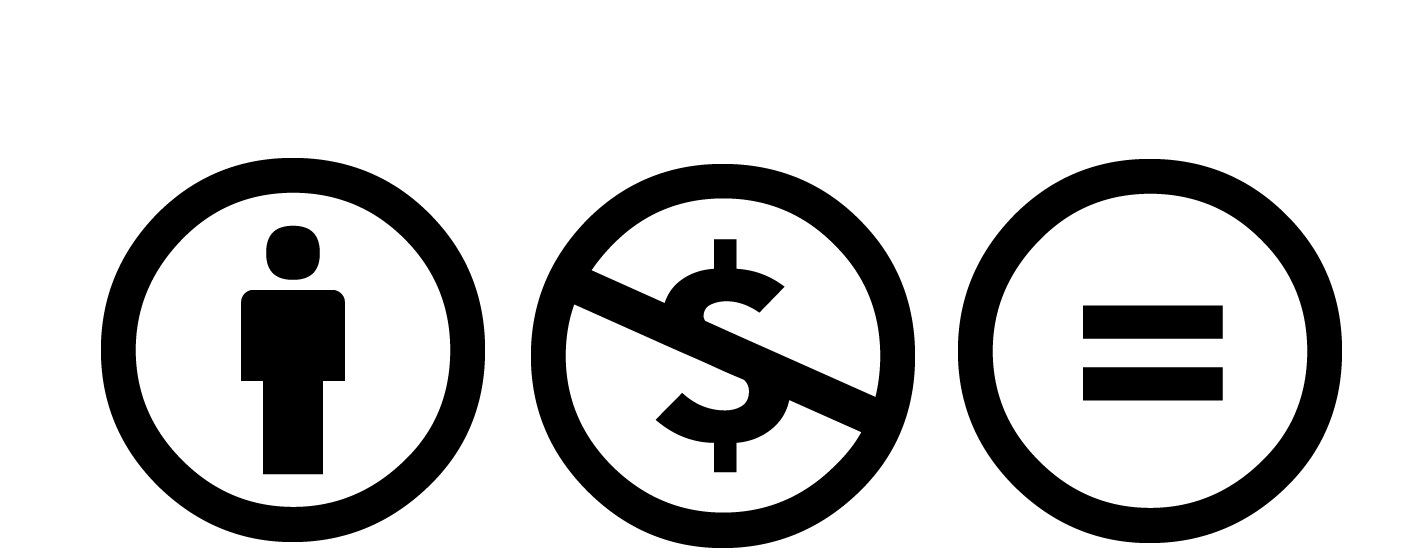}
\end{center}

\section{Introduction}
In this paper, we deal with the problem of combining two marginal
models representing the probabilistic information about two random
variables $X,Y$ into a bivariate model of the joint behaviour of
$(X,Y)$. In the classical case, this problem has a simple solution,
by means of {\em Sklar's} well-known {\em theorem} \cite{sklar1959},
that tells us that any bivariate distribution function can be
obtained as the combination of its marginals by means of a copula
\cite{nelsen1999}.

Here we investigate to what extent Sklar's theorem can be extended
in the context of imprecision, both in the marginal distribution
functions and in the copula that links them. The imprecision in
marginal distributions shall be modelled by a \emph{probability box}
\cite{ferson2003} (p-box, for short), that summarizes a set of
distribution functions by means of its lower and upper envelopes.
Regarding the imprecision about the copula, we shall also consider a
set of copulas. This set shall be represented by means of the newly
introduced notion of \emph{imprecise copula}, that we study in
Section~\ref{sec:bivariate-sklar}. This imprecision means that in
the bivariate case we end up with a set of bivariate distribution
functions, that we can summarize by means of a coherent bivariate
p-box, a notion recently studied in \cite{pelessoni2014}.

Interestingly, we shall show in Section~\ref{sec:bivariate-sklar}
that Sklar's theorem can be only partly extended to the imprecise
case; although the combination of two marginal p-boxes by means of a
set of copulas (or its associated imprecise copula) always produces
a coherent bivariate p-box, the most important aspect of the theorem
does not hold: not every coherent bivariate p-box can be obtained in
this manner. In Sections~\ref{sec:natex} and
\ref{sec:strong-product}, we consider two particular cases of
interest: that where we have no information about the copula that
links the two variables together, and that where we assume that the
two variables are independent. In those cases, we use Walley's
notions of natural extension \cite{walley1991} and (epistemic)
independent products \cite{cooman2011a,walley1991} to derive the
joint model.

In Section~\ref{sec:stochastic-orders}, we connect our results to
decision making by applying the notion of stochastic dominance
in this setting, and we establish a number of cases in which
the order existing on the marginals is preserved by their respective
joints. We conclude the paper with
some additional comments and remarks in Section~\ref{sec:conclusions}.

\section{Preliminary concepts}

\subsection{Coherent lower previsions}

Let us introduce the basic notions from the theory of coherent lower
previsions that we shall use later on in this paper. For a more
detailed exposition of the theory and for a behavioural
interpretation of the concepts below in terms of betting rates, we
refer to \cite{walley1991}.

Let $\Omega$ be a possibility space. A \emph{gamble} is a bounded
real-valued function $f:\Omega\rightarrow\mathbb{R}$. We shall
denote by $\gambles(\Omega)$ the set of all gambles on $\Omega$, and
by $\gambles^+(\Omega)$ the set of non-negative gambles. It includes
in particular the indicator functions of subsets $B$ of $\Omega$,
i.e., the gambles that take value $1$ on the elements of $B$ and $0$
elsewhere. In this paper, we shall use the same symbol for an event
$B$ and for its indicator function.

A \emph{lower prevision} is a functional
$\lpr:\domain\rightarrow\mathbb{R}$ defined on some set of gambles
$\domain\subseteq\gambles(\Omega)$. Here we are interested in lower
previsions satisfying the property of \emph{coherence}:

\begin{definition}[{\bf Coherent lower previsions}]
A lower prevision $\lpr:\gambles(\Omega)\rightarrow\mathbb{R}$ is
called \emph{coherent} when it satisfies the following conditions
for every pair of gambles $f,g\in\gambles(\Omega)$ and every
$\lambda>0$:
\begin{itemize}
 \item[(C1)] $\lpr(f)\geq\inf_{\omega\in \Omega} f(\omega)$.
 \item[(C2)] $\lpr(\lambda f)=\lambda \lpr(f)$.
 \item[(C3)] $\lpr(f+g)\geq\lpr(f)+\lpr(g)$.
\end{itemize}
\end{definition}

The restriction to events of a coherent lower prevision is called a
\emph{coherent lower probability}, and more generally a lower
prevision $\lpr$ on $\domain$ is said to be coherent whenever it can
be extended to a coherent lower prevision on $\gambles(\Omega)$. On
the other hand, if $\lpr$ is a coherent lower prevision on
$\gambles(\Omega)$ and it satisfies (C3) with equality for every $f$
and $g$ in $\gambles(\Omega)$, then it is called a \emph{linear}
prevision, and its restriction to events is a finitely additive
probability. In fact, coherent lower previsions can be given the
following sensitivity analysis interpretation: a lower prevision
$\lpr$ on $\domain$ is coherent if and only if it is the lower
envelope of its associated \emph{credal set},
\begin{equation}\label{eq:credal-set}
 \solp(\lpr):=\{\pr:\gambles(\Omega)\rightarrow\mathbb{R} \text{ linear prevision}: \pr(f)\geq\lpr(f) \ \forall
 f\in\domain\},
\end{equation}
and as a consequence the lower envelope of a set of linear
previsions is always a coherent lower prevision \cite[Section
3.3.3(b)]{walley1991}.

One particular instance of coherent lower probabilities are those
associated with $p$-boxes.

\begin{definition}\cite{ferson2003}\label{de:univ-pbox}
 A (univariate) \emph{$p$-box} is a pair $(\ldf,\udf)$ where
 $\ldf,\udf:\realsexp\rightarrow[0,1]$ are cumulative distribution
 functions (i.e., monotone and such that $\ldf(-\infty)=\udf(-\infty)=0,\ldf(+\infty)=\udf(+\infty)=1$) satisfying $\ldf(x)\leq\udf(x)$ for every
 $x\in\realsexp$.
\end{definition}

Define the set $A_x=[-\infty,x]$ for every $x\in\realsexp$, and let
\begin{equation*}
\mathcal{E}_{0}:=\{A_x: x\in\realsexp\}\cup \{A_x^c:
x\in\realsexp\}.
\end{equation*}
Then \cite{troffaes2011} a $p$-box $(\ldf,\udf)$ induces a coherent
lower probability $\lpr_{(\ldf,\udf)}: \mathcal{E}_{0}\rightarrow
[0,1]$ by
\begin{equation}\label{eq:lpr-from-pbox-univ}
 \lpr_{(\ldf,\udf)}(A_x)=\ldf(x) \text{ and }
 \lpr_{(\ldf,\udf)}(A_x^c)=1-\udf(x) \ \forall
 x\in\realsexp.
\end{equation}


\subsection{Bivariate $p$-boxes}

In \cite{pelessoni2014}, the notion of $p$-box from
Definition~\ref{de:univ-pbox} has been extended to the bivariate case,
to describe couples of random variables $(X,Y)$ in presence of imprecision.

\begin{definition}\cite{pelessoni2014}
A map $\df:\realsexp\times\realsexp\rightarrow [0,1]$ is called
\emph{standardized} when it is component-wise increasing, that is,
$\df(t_1,z)\leq \df(t_2,z)$ and $\df(z,t_1)\leq \df(z,t_2)$ for all
$t_1\leq t_2$ and $z$, and satisfies
\begin{equation*}
 \df(-\infty,y)=\df(x,-\infty)=0 \ \forall x,y\in\realsexp,\ \df(+\infty,+\infty)=1.
\end{equation*}
It is called a \emph{distribution function} for $(X,Y)$ when it is standardized
and satisfies
\begin{equation*}
 \df(x_2,y_2)+\df(x_1,y_1)-\df(x_1,y_2)-\df(x_2,y_1)\geq 0
\end{equation*}
for all $x_1, x_2, y_1, y_2\in\realsexp$ such that $x_1\leq x_2$, $y_1\leq y_2$
(with equality holding whenever $(x_1\leq X< x_2)\wedge (y_1\leq Y< y_2)$ is impossible).
Given two standardized functions
$\ldf,\udf:\realsexp\times\realsexp\rightarrow [0,1]$
satisfying $\ldf(x,y)\leq\udf(x,y)$ for every
$x,y\in\realsexp$, the pair $(\ldf,\udf)$ is called a {\em bivariate
$p$-box}.
\end{definition}

Bivariate $p$-boxes are introduced as a model for the imprecise
knowledge of a bivariate distribution function. The reason why the
lower and upper functions in a bivariate $p$-box are not required to
be distribution functions is that the lower and upper envelopes of a
set of bivariate distribution functions need not be distribution
functions themselves, as showed in \cite{pelessoni2014}.

Let $(\ldf,\udf)$ be a bivariate $p$-box on
$\realsexp\times\realsexp$. Define
$A_{(x,y)}=[-\infty,x]\times[-\infty,y]$ for every
$x,y\in\realsexp$, and consider the sets
\begin{equation*}
 {\mathcal D}:=\{A_{(x,y)}: x,y\in\realsexp\},
 {\mathcal D}_c:=\{A_{(x,y)}^c: x,y\in\realsexp\},
 \ \mathcal{E}:={\mathcal D}\cup{\mathcal D}_c.
\end{equation*}
Note that $A_{(+\infty,+\infty)}=\realsexp\times\realsexp$, whence
both $\realsexp\times\realsexp$ and $\emptyset$ belong to ${\mathcal
E}$. Similarly to Eq.~\eqref{eq:lpr-from-pbox-univ}, we can define
the \emph{lower probability induced by a bivariate $p$-box
$(\ldf,\udf)$} on $\realsexp\times\realsexp$ as the map
 $\lpr_{(\ldf,\udf)}:{\mathcal E}\rightarrow [0,1]$ given by:
\begin{equation}\label{eq:lpr-from-pbox}
 \lpr_{(\ldf,\udf)}(A_{(x,y)})=\ldf(x,y),\ \ \lpr_{(\ldf,\udf)}(A_{(x,y)}^c)=1-\udf(x,y)
\end{equation}
for every $x,y\in\realsexp$. Conversely, a lower probability
$\lpr:\mathcal{E}\rightarrow [0,1]$ determines a couple of functions
$\ldf_{\lpr},\udf_{\lpr}:\realsexp\times\realsexp\rightarrow[0,1]$
defined by
\begin{equation}\label{eq:pbox-from-lpr}
 \ldf_{\lpr}(x,y)=\lpr(A_{(x,y)}) \text{ and } \udf_{\lpr}(x,y)=1-\lpr(A_{(x,y)}^c) \ \forall x,y\in\realsexp.
\end{equation}
Then $(\ldf_{\lpr},\udf_{\lpr})$ is a bivariate $p$-box as soon as
the lower probability $\lpr$ is $2$-coherent \cite{pelessoni2014}.
$2$-coherence is a weak rationality condition implied by coherence
\cite[Appendix~B]{walley1991}, which in the context of this paper,
where the domain $\mathcal{E}$ is closed under complementation, is
equivalent \cite{pelessoni2014} to $\lpr_{(\ldf,\udf)}$ being
monotone, normalised, and such that
$\lpr_{(\ldf,\udf)}(E)+\lpr_{(\ldf,\udf)}(E^c)\leq 1$ for every
$E\in{\mathcal E}$.

The correspondence between bivariate $p$-boxes and lower
probabilities in terms of precise models is given by the following
lemma:\footnote{We give a brief sketch of the proof: it suffices to
establish the equivalences $\pr(A_{(x,y)}) \geq
\lpr_{(\ldf,\udf)}(A_{(x,y)}) \iff \df_{\pr}(x,y) \geq \ldf(x,y)$
and $\pr(A_{(x,y)}^c) \geq \lpr_{(\ldf,\udf)}(A_{(x,y)}^c) \iff
\df_{\pr}(x,y) \leq \udf(x,y)$
 for every $x,y\in\realsexp$. These follow easily from Eqs.~\eqref{eq:lpr-from-pbox} and~\eqref{eq:pbox-from-lpr}.}

\begin{lemma}\cite{pelessoni2014}
\label{le:dominating sets-correspondence} Let $(\ldf,\udf)$ be a
$p$-box  and $\lpr_{(\ldf,\udf)}$ the lower probability it induces
on $\mathcal{E}$ by means of Eq.~\eqref{eq:lpr-from-pbox}.
\begin{itemize}
 \item[(a)] Let $P$ be (the restriction to $\mathcal{E}$ of) a linear prevision on
 $\gambles(\realsexp\times\realsexp)$, and let $F_P$ be its associated
 distribution function given by $F_P(x,y)=P(A_{(x,y)})$ for every
 $x,y\in\realsexp$. Then
 \begin{equation*}
  P(A)\geq\lpr_{(\ldf,\udf)}(A) \ \forall A\in{\mathcal E} \iff \ldf\leq F_P\leq \udf.
 \end{equation*}
 \item[(b)] Conversely, let $F$ be a distribution function on $\realsexp\times\realsexp$, and let $\pr_{\df}:{\mathcal E}\rightarrow [0,1]$
 be the functional given by $\pr_{\df}(A_{(x,y)})=\df(x,y),
 \pr_{\df}(A_{(x,y)}^c)=1-\df(x,y)$ for every $x,y\in\realsexp$.
 Then $$\ldf\leq F\leq\udf \iff P_F(A)\geq\lpr_{(\ldf,\udf)}(A) \ \forall A\in{\mathcal E}.$$
\end{itemize}
\end{lemma}

Given a bivariate $p$-box $(\ldf,\udf)$,
Lemma \ref{le:dominating sets-correspondence} implies that the coherence
of its associated lower probability $\lpr_{(\ldf,\udf)}$
can be characterised through a set of distribution functions:

\begin{proposition}\cite{pelessoni2014}\label{pr:coherence-correspondence}
The lower probability $\lpr_{(\ldf,\udf)}$ induced by the bivariate
$p$-box $(\ldf,\udf)$ by means of Eq.~\eqref{eq:lpr-from-pbox} is
coherent if and only if $\ldf$ (resp., $\udf$) is the lower (resp.,
upper) envelope of the set
\begin{equation}
\label{eq:interval_distributions}
 {\mathcal{F}}=\{\df:\realsexp\times\realsexp\rightarrow [0,1] \text{ distribution function}: \ldf\leq\df\leq\udf\}.
\end{equation}
If $\lpr_{(\ldf,\udf)}$ is coherent, the following conditions hold
for every $x_1 \leq x_2\in\realsexp$ and $y_1\leq y_2\in\realsexp$:
\begin{align}
\label{eq:cond1}\ldf(x_2,y_2)+\udf(x_1,y_1)-\ldf(x_1,y_2)-\ldf(x_2,y_1)\geq0.\tag{I-RI1}\\
\label{eq:cond2}\udf(x_2,y_2)+\ldf(x_1,y_1)-\ldf(x_1,y_2)-\ldf(x_2,y_1)\geq0.\tag{I-RI2}\\
\label{eq:cond3}\udf(x_2,y_2)+\udf(x_1,y_1)-\udf(x_1,y_2)-\ldf(x_2,y_1)\geq0.\tag{I-RI3}\\
\label{eq:cond4}\udf(x_2,y_2)+\udf(x_1,y_1)-\ldf(x_1,y_2)-\udf(x_2,y_1)\geq0.\tag{I-RI4}
\end{align}
\end{proposition}

\begin{definition}
\label{def:coherence_p-box} A bivariate $p$-box $(\ldf,\udf)$ is
\emph{coherent} whenever the lower probability $\lpr_{(\ldf,\udf)}$
it induces on $\mathcal{E}$ by means of Eq.~\eqref{eq:lpr-from-pbox}
is coherent.
\end{definition}

\subsection{Copulas}

In this paper, we are going to study to what extent bivariate
$p$-boxes can be expressed as a function of their marginals. In the
precise case (that is, when we have only one bivariate distribution
function), this is done through the notion of copula.

\begin{definition}\cite{nelsen1999}\label{de:copula}
A function $C:[0,1]\times[0,1]\rightarrow [0,1]$ is called a {\em
copula} when it satisfies the following conditions:
\begin{align}{}
 \label{eq:cop_1} &C(0,u)=C(u,0)=0\ \forall u\in[0,1] \tag{COP1}.\\
 \label{eq:cop_2} &C(1,u)=C(u,1)=u\ \forall u\in[0,1]\tag{COP2}.\\
 \label{eq:cop_3} &C(u_2,v_2)+C(u_1,v_1)-C(u_1,v_2)-C(u_2,v_1)\geq 0\ \forall
 u_1\leq u_2, v_1\leq v_2\in [0,1]\tag{COP3}.
\end{align}
\end{definition}

It follows from the definition above that a copula is component-wise
monotone increasing. One of the main features of copulas lies in
Sklar's theorem:

\begin{theorem}[\cite{sklar1959}{{\bf , Sklar's Theorem}}]\label{th:sklar-precise}
Let $\df_{\rm (X,Y)}:\realsexp\times\realsexp\rightarrow [0,1]$ be a
bivariate distribution function with marginals $\df_{\rm
X}:\realsexp\rightarrow[0,1]$ and $\df_{\rm
Y}:\realsexp\rightarrow[0,1]$, defined by $\df_{\rm X}(x)=\df_{\rm
(X,Y)}(x,+\infty)$ and $\df_{\rm Y}(y)=\df_{\rm (X,Y)}(+\infty,y)$
for any $x$ and $y$ in $\realsexp$. Then there is a copula $C$ such
that
\[
F_{\rm (X,Y)}(x,y)=C(F_{\rm X}(x),F_{\rm Y}(y)) \mbox{ for all
}(x,y)\in\realsexp\times\realsexp.
\]
Conversely, any transformation of marginal distribution functions by
means of a copula produces a bivariate distribution function.
\end{theorem}

Any copula $C$ must satisfy the {\em Fr\'{e}chet-Hoeffding bounds}
(see \cite{frechet1935,williamson1989}):
\begin{equation}\label{eq:Frechet-bounds}
 C_{\rm L}(u,v):=\max\{u+v-1,0\}\leq C(u,v)\leq \min\{u,v\}:=C_{\rm M}(u,v)
\end{equation}
for every $u,v\in[0,1]$. $C_{\rm L}$ is called the {\em
{\L}ukasiewicz copula} and $C_{\rm M}$ the {\em minimum copula}.
Eq.~\eqref{eq:Frechet-bounds} applies in particular to one instance
of copulas that shall be of interest in this paper: the {\em product
copula} $C_{\rm P}$, given by $C_{\rm P}(u,v)=u\cdot v$ for every
$u,v\in[0,1]$. It holds that two random variables $X,Y$ are
stochastically independent if and only if their distribution functions are coupled
by the product copula.

For an in-depth review on copulas we refer to \cite{nelsen1999}.

\section{Combining marginal $p$-boxes into a bivariate one}

One particular context where bivariate $p$-boxes can arise is in the
joint extension of two marginal $p$-boxes. In this section, we
explore this case in detail, studying in particular the properties
of some bivariate $p$-boxes with given marginals: the largest one,
that shall be obtained by means of the Fr\'{e}chet bounds and the
notion of natural extension, and the one modelling the notion of
independence. In both cases, we shall see that the bivariate model
can be derived by means of an appropriate extension of the notion of
copula.

Related results can be found in \cite[Section~7]{troffaes2011}, with
one fundamental difference: in \cite{troffaes2011}, the authors use
the existence of a total preorder on the product space (in the case
of this paper, $\realsexp\times\realsexp$) that is compatible with
the orders in the marginal spaces, and reduce the multivariate p-box
to a univariate one. Here we do no such reduction, and we consider
only a partial order: the product order, given by
\begin{equation*}
 (x_1,y_1) \leq (x_2,y_2) \Leftrightarrow x_1 \leq x_2 \text{ and }
 y_1 \leq y_2.
\end{equation*}
Another related study was made by Yager in \cite{Yager13},
considering the case in which the marginal distributions are not
precisely described and are defined by means of Dempster-Shafer
belief structures instead. He modelled this situation by considering
copulas whose arguments are intervals (the ones determined by the
Demspter-Shafer models) instead of crisp numbers, and whose images
are also intervals. He showed then that the lower (resp., upper)
bound of the interval of images corresponds to the copula evaluated
in the lower (resp., upper) bounds of the intervals. This can be
seen as a particular case of our subsequent
Proposition~\ref{prop:Imp_Sklar2}.

\subsection{A generalization of Sklar's theorem}\label{sec:bivariate-sklar}
Let us study to which extent Sklar's theorem can be generalised to a
context of imprecision, both in the marginal distribution functions
to be combined and in the copula that links them. 
In order to tackle this problem, we introduce the notion of
imprecise copula:

\begin{definition}\label{def:impr_cop}
A pair $(\lcop,\ucop)$ of functions
$\lcop,\ucop:[0,1]\times[0,1]\rightarrow[0,1]$ is called an {\em
imprecise copula} if:
\begin{itemize}
\item $\lcop(0,u)=\lcop(u,0)=0$, $\lcop(1,u)=\lcop(u,1)=u$ $\forall u\in[0,1]$.
\item $\ucop(0,u)=\ucop(u,0)=0$, $\ucop(1,u)=\ucop(u,1)=u$ $\forall u\in[0,1]$.
\item For any $u_1\leq u_2, v_1\leq v_2$:
\begin{align}{}
\label{eq:ineq_cop_1}\lcop(u_2,v_2)+\ucop(u_1,v_1)-\lcop(u_1,v_2)-\lcop(u_2,v_1)\geq 0.\tag{CI-1}\\
\label{eq:ineq_cop_2}\ucop(u_2,v_2)+\lcop(u_1,v_1)-\lcop(u_1,v_2)-\lcop(u_2,v_1)\geq 0.\tag{CI-2}\\
\label{eq:ineq_cop_3}\ucop(u_2,v_2)+\ucop(u_1,v_1)-\ucop(u_1,v_2)-\lcop(u_2,v_1)\geq 0.\tag{CI-3}\\
\label{eq:ineq_cop_4}\ucop(u_2,v_2)+\ucop(u_1,v_1)-\lcop(u_1,v_2)-\ucop(u_2,v_1)\geq 0.\tag{CI-4}
\end{align}
\end{itemize}
\end{definition}





We are using the terminology imprecise copula in the definition
above because we intend it as a mathematical model for the imprecise
knowledge of a copula; note however that the lower and upper
functions $\lcop,\ucop$ need not be copulas themselves, because they
may not
satisfy the 2-increasing property (COP3). 

\eqref{eq:ineq_cop_1}$\div$\eqref{eq:ineq_cop_4} are useful in
establishing the following properties of imprecise copulas.
\begin{proposition}
\label{pro:copula_properties} Let $(\lcop,\ucop)$ be an imprecise
copula.
\begin{itemize}
\item[(a)] $\lcop\leq\ucop$.
\item[(b)] $\lcop$ and $\ucop$ are component-wise increasing.
\item[(c)] The \emph{Lipschitz condition}
\begin{equation}
\label{eq:Lipschitz}
|C(u_2,v_2)-C(u_1,v_1)|\leq |u_2-u_1|+|v_2-v_1|
\ \forall u_1, u_2, v_1, v_2\in [0,1]
\end{equation}
is satisfied both by $C=\lcop$ and by $C=\ucop$.
\item[(d)] The pointwise infimum and supremum of a non-empty set of copulas $\mathcal{C}$ form an imprecise copula.
\end{itemize}
\end{proposition}
\begin{proof}
\begin{itemize}
\item[(a)] This follows from inequality \eqref{eq:ineq_cop_3}, with $u_2=u_1$.

\item[(b)] Use~\eqref{eq:ineq_cop_1} with,
alternatively, $v_1=0$ and $u_1=0$ to obtain, respectively,
\begin{align*}
\lcop(u_2,v_2)-\lcop(u_1,v_2)\geq 0& \ \forall v_2, u_1, u_2\in [0,1], s.t.\ u_1\leq u_2\\
\lcop(u_2,v_2)-\lcop(u_2,v_1)\geq 0& \ \forall u_2, v_1, v_2\in [0,1], s.t.\ v_1\leq v_2
\end{align*}
By these inequalities, $\lcop$ is component-wise increasing.
Analogously, to prove that $\ucop$ is component-wise increasing,
apply \eqref{eq:ineq_cop_4} with $u_1=0$ and \eqref{eq:ineq_cop_3}
with $v_1=0$.

\item[(c)] Applying twice \eqref{eq:ineq_cop_2}
and the boundary conditions in Definition \ref{de:copula},
first with $v_2=1$ and then with $u_2=1$,
we obtain, respectively,
\begin{align}
\label{eq:Lipschitz_intermediate_1}
\lcop(u_2,v_1)-\lcop(u_1,v_1)\leq u_2-u_1\\
\label{eq:Lipschitz_intermediate_2}
\lcop(u_1,v_2)-\lcop(u_1,v_1)\leq v_2-v_1.
\end{align}
Because they are derived from \eqref{eq:ineq_cop_2},
Eqs.~\eqref{eq:Lipschitz_intermediate_1}
and~\eqref{eq:Lipschitz_intermediate_2} hold, respectively, for any
$v_1, u_1, u_2\in [0,1]$ such that $u_1\leq u_2$, and for any $u_1,
v_1, v_2\in [0,1]$ such that $v_1\leq v_2$. In the general case,
Eq.~\eqref{eq:Lipschitz_intermediate_1} is replaced by
\begin{eqnarray*}
\label{eq:Lipschitz_intermediate_general_1}
|\lcop(u_2,v_1)-\lcop(u_1,v_1)|\leq |u_2-u_1|
\end{eqnarray*}
and similarly for Eq.~\eqref{eq:Lipschitz_intermediate_2}.
Therefore, for arbitrary $u_1, u_2, v_1$ and $v_2$ in $[0,1]$,
$|\lcop(u_2,v_2)-\lcop(u_1,v_1)|\leq
|\lcop(u_2,v_2)-\lcop(u_2,v_1)|+|\lcop(u_2,v_1)-\lcop(u_1,v_1)|\leq
|u_2-u_1|+|v_2-v_1|$, which proves the Lipschitz condition for
$\lcop$. The proof for $\ucop$ is similar (use \eqref{eq:ineq_cop_4}
with $v_2=1$ and \eqref{eq:ineq_cop_3} with $u_2=1$).

\item[(d)] The boundary conditions are trivial, so let us prove~\eqref{eq:ineq_cop_1}$\div$\eqref{eq:ineq_cop_4}.
Define $\lcop(x,y):=\inf_{C\in\mathcal{C}}C(x,y)$,
$\ucop(x,y):=\inf_{C\in\mathcal{C}}C(x,y)$. By applying
\eqref{eq:cop_3} to the copulas in $\mathcal{C}$, we get that, for
every $C\in\mathcal{C}$ and every $u_1\leq u_2, v_1\leq v_2\in
[0,1]$,
\begin{equation*}
C(u_2,v_2)+C(u_1,v_1)\geq C(u_1,v_2)+C(u_2,v_1)\geq \lcop(u_1,v_2)+\lcop(u_2,v_1).
\end{equation*}
From this we deduce that, for every $C\in\mathcal{C}$ and every
$u_1\leq u_2, v_1\leq v_2\in [0,1]$,
\begin{eqnarray*}
\ucop(u_2,v_2)+C(u_1,v_1)\geq\lcop(u_1,v_2)+\lcop(u_2,v_1), \\
C(u_2,v_2)+\ucop(u_1,v_1)\geq\lcop(u_1,v_2)+\lcop(u_2,v_1),
\end{eqnarray*}
whence \eqref{eq:ineq_cop_2} and \eqref{eq:ineq_cop_1} hold.

As for \eqref{eq:ineq_cop_3} and \eqref{eq:ineq_cop_4}, again from
\eqref{eq:cop_3}, we get that for every $C\in\mathcal{C}$ and every
$u_1\leq u_2, v_1\leq v_2\in [0,1]$,
\begin{equation*}
\label{eq:envelope_copule_3}
\ucop(u_2,v_2)+\ucop(u_1,v_1)\geq C(u_2,v_2)+C(u_1,v_1)\geq C(u_1,v_2)+C(u_2,v_1).
\end{equation*}
This implies that for every $C\in\mathcal{C}$ and every $u_1\leq
u_2, v_1\leq v_2\in [0,1]$,
\begin{eqnarray*}
\ucop(u_2,v_2)+\ucop(u_1,v_1)\geq C(u_1,v_2)+\lcop(u_2,v_1), \\
\ucop(u_2,v_2)+\ucop(u_1,v_1)\geq\lcop(u_1,v_2)+C(u_2,v_1),
\end{eqnarray*}
whence~\eqref{eq:ineq_cop_3} and~\eqref{eq:ineq_cop_4} hold.
$\qedhere$
\end{itemize}
\end{proof}
According to \cite[Corollary~2.3]{Nelsen04}, the pointwise infimum
and supremum of a set of copulas are also quasi-copulas (see
\cite{Nelsen05} for a study on the lattice structure of copulas). A
quasi-copula \cite{nelsen1999} is a binary operator satisfying
conditions (COP1), (COP2) in Definition~\ref{de:copula} and the
Lipschitz condition given by Eq.~\eqref{eq:Lipschitz}.

By Proposition \ref{pro:copula_properties}~(c), both $\lcop$ and
$\ucop$ in an imprecise copula $(\lcop,\ucop)$ are quasi-copulas.
Conversely, given two quasi-copulas $C_1$ and $C_2$ such that
$C_1\leq C_2$, $(C_1,C_2)$ may not be an imprecise copula. To see
that, it is enough to consider a proper quasi-copula $C$, i.e. a
quasi-copula which is not a copula (see for instance
\cite[Example~6.3]{nelsen1999}). Then, the pair $(C,C)$ is not an
imprecise copula because it does not satisfy the inequalities in
Definition~\ref{def:impr_cop}: in this case the inequalities all
reduce to \eqref{eq:cop_3}. We may then conclude that an imprecise
copula is formed by two quasi-copulas $C_1\leq C_2$, for which the
additional inequalities
\eqref{eq:ineq_cop_1}$\div$\eqref{eq:ineq_cop_4} hold.

The converse of item~(d) in this proposition is still an open
problem at this stage; it is formally equivalent to the
characterisation of coherent bivariate $p$-boxes studied in
detail in \cite{pelessoni2014}. So far, we have only established it
under some restrictions on the domains of the copulas. If it held,
then we could regard imprecise copulas as restrictions of
\emph{sets} of bivariate distribution functions of continuous random
variables with uniform marginals, similar to the situation for
precise copulas.

In the particular case when $\lcop=\ucop:=C$,
$(\lcop,\ucop)$ is an imprecise copula if and only if $C$ is a
copula. It is also immediate to establish the following:

\begin{proposition}\label{pr:largest-imprecise-copula}
Let $C_1$ and $C_2$ be two copulas such that $C_1\leq C_2$. Then,
$(C_1,C_2)$ forms an imprecise copula. In particular, $(C_{\rm
L},C_{\rm M})$ is the largest imprecise copula, in the sense that,
for any imprecise copula $(\lcop,\ucop)$, it holds that $C_{\rm
L}\leq \lcop\leq \ucop\leq C_{\rm M}$.
\end{proposition}
\begin{proof}
It is simple to check that $(C_1,C_2)$ satisfies Definition
\ref{def:impr_cop}. The proof of the remaining part is similar to
that of the Fr\'{e}chet-Hoeffding inequalities. Consider an
imprecise copula $(\lcop,\ucop)$. Since $\ucop$ is component-wise
increasing by Proposition \ref{pro:copula_properties}~(b), and
applying the boundary conditions,
\[
\begin{array}{l}
\ucop(u,v)\leq \min(\ucop(u,1),\ucop(1,v))=\min(u,v).
\end{array}
\]
Using~\eqref{eq:ineq_cop_2} we deduce that:
\[
1+\lcop(u,v)=\ucop(1,1)+\lcop(u,v)\geq \lcop(u,1)+\lcop(1,v)=u+v.
\]
Then $\lcop(u,v)\geq u+v-1$, and by definition $\lcop$ is also
non-negative. Finally, the inequality $\lcop\leq\ucop$ follows from
Proposition \ref{pro:copula_properties}~(a).
\end{proof}


\begin{remark}\label{rem:copula}
Given a copula $C:[0,1]\times[0,1]\rightarrow [0,1]$, it is
immediate to see that its extension
$C':\realsexp\times\realsexp\rightarrow [0,1]$ given by
\begin{equation*}
 C'(x,y):=\begin{cases}
  C(x,y) &\text{ if } (x,y)\in[0,1]\times[0,1] \\
  0 &\text{ if } x<0 \text{ or } y<0 \\
  \min\{x,y\} &\text{ if } \min\{x,y\}\in [0,1] \text{ and } \max\{x,y\}\in[1,+\infty]\\
  1 &\text{ otherwise}
  \end{cases}
\end{equation*}
is a distribution function. Taking this into account, given any
non-empty set of copulas $\C$, its infimum $\lcop$ and supremum
$\ucop$ form a coherent bivariate $p$-box. Moreover, an imprecise
copula $(\lcop,\ucop)$ can be extended to $\realsexp\times\realsexp$
in the manner described above, and then it constitutes a bivariate
$p$-box that satisfies
conditions~\eqref{eq:cond1}$\div$\eqref{eq:cond4} (although it is
still an open problem whether it is coherent).
$\blacklozenge$
\end{remark}

Let us see to what extent an analogue of Sklar's theorem also holds
in an imprecise framework. For this aim, we start by considering
marginal imprecise distributions, described by (univariate)
$p$-boxes, and we use imprecise copulas to obtain a bivariate
$p$-box.

\begin{proposition}\label{prop:Imp_Sklar2}
Let $(\ldf_{\rm X},\udf_{\rm X})$ and $(\ldf_{\rm Y},\udf_{\rm Y})$
be two marginal $p$-boxes on $\realsexp$, and let $\mathcal{C}$ be a
set of copulas. Consider the imprecise copula $(\lcop,\ucop)$
defined from $\mathcal{C}$ by
$\lcop(u,v)=\inf_{C\in\mathcal{C}}C(u,v)$ and
$\ucop(u,v)=\sup_{C\in\mathcal{C}}C(u,v)$ for every $u,v\in[0,1]$.
Define the couple $(\ldf,\udf)$ by:
\begin{equation}\label{eq:impr-pbox-envelope}
\ldf(x,y)=\lcop(\ldf_{\rm X}(x),\ldf_{\rm Y}(y))
\mbox{ and }\udf(x,y)=\ucop(\udf_{\rm
X}(x),\udf_{\rm Y}(y))
\end{equation}
for any $(x,y)\in\realsexp\times\realsexp$.
Then, $(\ldf,\udf)$ is a bivariate $p$-box and it holds that:
\begin{itemize}
\item[(a)] $\lpr_{(\ldf,\udf)}$ is coherent.

\item[(b)] The credal set $\mathcal{M}(\lpr_{(\ldf,\udf)})$ associated with the lower
probability $\lpr_{(\ldf,\udf)}$ by means of
Eq.~\eqref{eq:credal-set} is given by \[
\mathcal{M}(\lpr_{(\ldf,\udf)})=\{ P \mbox{ linear prevision}\mid
\lcop(\ldf_{\rm X},\ldf_{\rm Y})\leq F_P\leq \ucop(\udf_{\rm X},\udf_{\rm Y})\}.
\]
\end{itemize}
\end{proposition}

\begin{proof}
Note that $\ldf\leq\udf$, since
$\ldf=\lcop(\ldf_X,\ldf_Y)\leq\lcop(\udf_X,\udf_Y)\leq\ucop(\udf_X,\udf_Y)=\udf$.
It is easy to check that both $\ldf,\udf$ are standardized and as a
consequence $(\ldf,\udf)$ is a bivariate $p$-box.
\begin{itemize}
\item[(a)] Let $\dfs$ be the set of distribution functions associated with the
bivariate $p$-box $(\ldf,\udf)$ by means of
Eq.~\eqref{eq:interval_distributions}. Since $\ldf_{\rm X},\udf_{\rm
X},\ldf_{\rm Y},\udf_{\rm Y}$ are marginal distribution functions,
Sklar's theorem implies that $C(\ldf_{\rm X}(x),\ldf_{\rm Y}(y))$
and $C(\udf_{\rm X}(x),\udf_{\rm Y}(y))$ are bivariate distribution
functions for any $C\in\mathcal{C}$. Moreover, they necessarily
belong to $\mathcal{F}$ by Eq.~\eqref{eq:impr-pbox-envelope}. From
this we deduce that
\begin{equation*}
\ldf(x,y)\leq\inf_{F\in\mathcal{F}} F(x,y)\leq\lcop(\ldf_{\rm
X}(x),\ldf_{\rm Y}(y))=\ldf(x,y),
\end{equation*}
and therefore $\ldf(x,y)=\inf_{F\in\mathcal{F}} F(x,y)$. Similarly,
we can prove that $\udf(x,y)=\sup_{F\in\mathcal{F}} F(x,y)$.
Applying now Proposition~\ref{pr:coherence-correspondence}, we
deduce that $\lpr_{(\ldf,\udf)}$ is coherent.

\item[(b)] This follows from the first statement and
Lemma~\ref{le:dominating sets-correspondence}. 
$\qedhere$
\end{itemize}
\end{proof}

In particular, when the available information about the marginal distributions
is precise, and it is given by the distribution functions $F_{\rm
X}$ and $F_{\rm Y}$, the bivariate $p$-box in the proposition above
is given by
\begin{equation*}
\ldf(x,y)=\inf_{C\in\mathcal{C}}C(F_{\rm X}(x),F_{\rm Y}(y)) \text{
and } \udf(x,y)=\sup_{C\in\mathcal{C}}C(F_{\rm X}(x),F_{\rm Y}(y))
\end{equation*}
for every $(x,y)\in\realsexp\times\realsexp$. As a consequence, the
result above generalizes \cite[Theorem~2.4]{Nelsen04}, where the
authors only focused on the functions $\ldf$ and $\udf$, showing
that $\ldf(x,y)=\lcop(F_{\rm X}(x),F_{\rm Y}(y))$ and
$\udf(x,y)=\ucop(F_{\rm X}(x),F_{\rm Y}(y))$. 
Instead, in Proposition~\ref{prop:Imp_Sklar2} we are also allowing
for the existence of imprecision in the marginal distributions, that
we model by means of $p$-boxes.
Note that we have also established the
coherence of the joint lower probability $\lpr_{(\ldf,\udf)}$
and therefore of the $p$-box $(\ldf,\udf)$.

%
%

Proposition~\ref{prop:Imp_Sklar2} generalizes to the imprecise case
one of the implications in Sklar's theorem: if we combine two
marginal $p$-boxes by means of a set of copulas, we obtain a
coherent bivariate $p$-box, which is thus equivalent to a set of
bivariate distribution functions. We focus now on the other
implication: whether any bivariate $p$-box can be obtained as a
function of its marginals.\footnote{A similar study was made in
\cite[Theorem~2.4]{Durante} in terms of capacities and semi-copulas,
showing that the survival functions induced by a capacity can always
be expressed as a semi-copula of their marginals. Here we
investigate when the combination can be made in terms of an
imprecise copula. Note moreover that our focus is on \emph{coherent}
bivariate $p$-boxes, which produces capacities that are most
restrictive than those considered in \cite{Durante} (they are closer
to the precise case, so to speak). This is why we also consider the
particular case where the semi-copulas constitute an imprecise
copula.}

A partial result in this sense has been established in
\cite[Theorem~9]{scarsini1996}. In our language, it ensures that if
the restriction on $\mathcal{D}$ of $\lpr_{(\ldf,\udf)}$ is (a
restriction of) a $2$-monotone lower probability, then there exists
a function $\lcop:[0,1]\times [0,1]\rightarrow [0,1]$, which is
component-wise increasing and satisfies \eqref{eq:cop_1} and
\eqref{eq:cop_2}, such that $\ldf(x,y)=\lcop(\ldf_{\rm
X}(x),\ldf_{\rm Y}(y))$ for every $(x,y)$ in
$\realsexp\times\realsexp$. This has been used in the context of
random sets in \cite{alvarez2009,schmelzer2014}.

Somewhat surprisingly, we show next that this result cannot be
generalized to arbitrary $p$-boxes.

\begin{example}\label{ex:converse-sklar-not}
Let $\pr_1,\pr_2$ be the discrete probability measures associated
with the following masses on $\mathcal{X}\times\mathcal{Y}=\{1,2,3\}\times\{1,2\}$:
\begin{center}
\begin{tabular}{|c|c|c|c|c|c|c|}
  \hline
  & $(1,1)$ & $(2,1)$ & $(1,2)$ & $(2,2)$ & $(3,1)$ & $(3,2)$ \\ \hline
  $\pr_1$ & 0.2 & 0 & 0.3 & 0 & 0 & 0.5 \\
  $\pr_2$ & 0.1 & 0.2 & 0.5 & 0.1 & 0 & 0.1 \\
  \hline
\end{tabular}
\end{center}
Let $\lpr$ be the lower envelope of $\{\pr_1,\pr_2\}$. Then, $\lpr$
is a coherent lower probability, and its associated $p$-box
$(\ldf,\udf)$ satisfies
\begin{equation*}
\ldf_{\rm X}(1)=\ldf_{\rm X}(2)=0.5, \ldf_{\rm Y}(1)=0.2, \ldf(1,1)=0.1 <
\ldf(2,1)=0.2.
\end{equation*}
If there was a function $\lcop$ such that
$\ldf(x,y)=\lcop(\ldf_{\rm X}(x),\ldf_{\rm Y}(y))$ for every
$(x,y)\in\realsexp\times\realsexp$, then we should have
\begin{equation*}
 \ldf(1,1)=\lcop(\ldf_{\rm X}(1),\ldf_{\rm Y}(1))=\lcop(\ldf_{\rm X}(2),\ldf_{\rm Y}(1))=\ldf(2,1).
\end{equation*}
This is a contradiction. As a consequence, the lower distribution in
the bivariate $p$-box cannot be expressed as a function of its
marginals. $\blacklozenge$
\end{example}

This shows that the direct implication of Sklar's theorem does not
hold in the bivariate case: given a coherent bivariate $p$-box
$(\ldf,\udf)$, there is not in general an imprecise copula
$(\lcop,\ucop)$ determining it by means of
Eq.~\eqref{eq:impr-pbox-envelope}. The key point here is that the
lower and upper distribution functions of a coherent bivariate
$p$-box may not be distribution functions themselves, as showed in
\cite{pelessoni2014}; they need only be standardized functions.
Indeed, if $\ldf,\udf$ were distribution functions we could always
apply Sklar's theorem to them, and we could express each of them as
a copula of its marginals. What Example~\ref{ex:converse-sklar-not}
shows is that this is no longer possible when $\ldf,\udf$ are just
standardized functions, nor in general when $(\ldf,\udf)$ is
coherent. We can thus summarize the results of this section in the
following theorem:

\begin{theorem}[{\bf Imprecise Sklar's Theorem}]
The following statements hold:
\begin{itemize}
\item[(a)] Given two marginal p-boxes $(\ldf_{\rm X},\udf_{\rm X})$ and
$(\ldf_{\rm Y},\udf_{\rm Y})$ on $\realsexp$ and a set of copulas
$\mathcal{C}$, the functions $\ldf,\udf$ given by
Eq.~\eqref{eq:impr-pbox-envelope} determine a bivariate $p$-box on
$\realsexp\times\realsexp$, whose associated lower probability is
coherent.
\item[(b)] Not every bivariate $p$-box can be expressed by means of its marginals and a set of copulas by
Eq.~\eqref{eq:impr-pbox-envelope}, not even when its associated lower
probability is coherent.
\end{itemize}
\end{theorem}

\subsection{Natural extension of marginal $p$-boxes}\label{sec:natex}

Next we consider two particular combinations of the marginal
$p$-boxes into the bivariate one. First of all, we consider the case
where there is no information about the copula that links the
marginal distribution functions.

\begin{lemma}
Consider the univariate $p$-boxes $(\ldf_{\rm X},\udf_{\rm X})$ and
$(\ldf_{\rm Y},\udf_{\rm Y})$ on $\realsexp$, and let $\lpr$ be the lower
probability defined on
\begin{equation*}
\mathcal{A}^{*}:=\{A_{(x,+\infty)}, A_{(x,+\infty)}^c,
A_{(+\infty,y)}, A_{(+\infty,y)}^c:
x,y\in\realsexp\}\subseteq{\mathcal E}
\end{equation*}
by
\begin{eqnarray}
\lpr(A_{(x,+\infty)})=\ldf_{\rm X}(x) \qquad \lpr(A_{(x,+\infty)}^c)=1-\udf_{\rm X}(x) \ \forall x\in\realsexp,\label{eq:def-marginal-lpr-x}\\
\lpr(A_{(+\infty,y)})=\ldf_{\rm Y}(y) \qquad
\lpr(A_{(+\infty,y)}^c)=1-\udf_{\rm Y}(y) \ \forall
y\in\realsexp.\label{eq:def-marginal-lpr-y}
\end{eqnarray}
Then:
\begin{enumerate}
\item $\lpr$ is a coherent lower probability.
\item $\mathcal{M}(\lpr)=\mathcal{M}(C_{\rm L},C_{\rm M})$, where $C_{\rm L},C_{\rm M}$ are the copulas given by
Eq.~\eqref{eq:Frechet-bounds} and
\[
\mathcal{M}(C_{\rm L},C_{\rm M})=\{P \mbox{ linear prevision}:
C_{\rm L}(\ldf_{\rm X},\ldf_{\rm Y})\leq F_{\rm P}\leq C_{\rm
M}(\udf_{\rm X},\udf_{\rm Y}) \}.
\]
\end{enumerate}
\end{lemma}
\begin{proof}
\begin{enumerate}
\item We use $\mathcal{C}^{*}$ to denote
the set of all copulas. By Propositions~\ref{pr:largest-imprecise-copula} and \ref{prop:Imp_Sklar2} and
Eq.~\eqref{eq:Frechet-bounds},
\begin{align*}
 \ldf(x,y)=\lcop(\ldf_{\rm X}(x),\ldf_{\rm Y}(y))&=\inf_{C\in{\mathcal{C}^{*}}}C(\ldf_{\rm X}(x),\ldf_{\rm Y}(y))\\&=C_{\rm L}(\ldf_{\rm X}(x),\ldf_{\rm Y}(y)),
\end{align*}
and similarly
$\udf(x,y)=\ucop(\udf_{\rm X}(x),\udf_{\rm Y}(y))=C_{\rm M}(\udf_{\rm X}(x),\udf_{\rm Y}(y))$.
Let $\lpr_{(\ldf,\udf)}$ the coherent lower probability induced by
$(\ldf,\udf)$ by Eq.~\eqref{eq:lpr-from-pbox}. Then
\begin{align*}
 \lpr_{(\ldf,\udf)}(A_{(x,+\infty)})&=\ldf(x,+\infty)=C_{\rm L}(\ldf_{\rm X}(x),\ldf_{\rm Y}(+\infty))\\&=\max\{\ldf_{\rm X}(x)+\ldf_{\rm Y}(+\infty)-1,0\}\\&=
\max\{\ldf_{\rm X}(x),0\}=\ldf_{\rm X}(x)=\lpr(A_{(x,+\infty)})
\end{align*}
and also
\begin{align*}
\lpr_{(\ldf,\udf)}(A_{(x,+\infty)}^c)&=1-\udf(x,+\infty)=1-C_{\rm M}(\udf_{\rm X}(x),\udf_{\rm Y}(+\infty))\\
&=1-\min\{\udf_{\rm X}(x),\udf_{\rm Y}(+\infty)\}=
1-\min\{\udf_{\rm X}(x),1\}\\&=1-\udf_{\rm X}(x)=\lpr(A_{(x,+\infty)}^c).
\end{align*}
With an analogous reasoning, we obtain
$\lpr_{(\ldf,\udf)}(A_{(+\infty,y)})=\lpr(A_{(+\infty,y)})$ and
$\lpr_{(\ldf,\udf)}(A_{(+\infty,y)}^c)=\lpr(A_{(+\infty,y)}^c)$.
Therefore, $\lpr$ coincides with $\lpr_{(\ldf,\udf)}$ in
$\mathcal{A}^{*}$, and consequently $\lpr$ is coherent.
\item
Let $P\in\mathcal{M}(C_{\rm L},C_{\rm M})$. Then,
$P\geq\lpr_{(\ldf,\udf)}$ on $\mathcal{E}$ by
Lemma~\ref{le:dominating sets-correspondence}. Since $\lpr$
coincides with $\lpr_{(\ldf,\udf)}$ on $\mathcal{A}^{*}$,
$P\in\mathcal{M}(\lpr)$.

Conversely, let $P\in\mathcal{M}(\lpr)$, and let $F_{\rm P}$ be its
associated distribution function. Then, Sklar's Theorem assures that
there is $C\in\mathcal{C}^{*}$ such that $F_{\rm P}(x,y)=C(F_{\rm
P}(x,+\infty),F_{\rm P}(+\infty,y))$ for every
$(x,y)\in\realsexp\times\realsexp$. Hence,
    \[
    \begin{array}{l c l}
    C_{\rm L}(\ldf_X(x),\ldf_Y(y))&\leq& C_{\rm L}(F_{\rm P}(x,+\infty),F_{\rm P}(+\infty,y))\\
                            &\leq& C(F_{\rm P}(x,+\infty),F_{\rm P}(+\infty,y))\\
                            &\leq& C(\udf_X(x),\udf_Y(y))\leq C_{\rm M}(\udf_X(x),\udf_Y(y)),
    \end{array}
    \]
taking into account that any copula is component-wise increasing and
lies between $C_{\rm L}$ and $C_{\rm M}$. Therefore,
$P\in\mathcal{M}(C_{\rm L},C_{\rm M})$ and as a consequence
$\mathcal{M}(\lpr)=\mathcal{M}(C_{\rm L},C_{\rm M})$. $\qedhere$
\end{enumerate}
\end{proof}

From this result we can immediately derive the expression of the
{\em natural extension} \cite{walley1991} of two marginal $p$-boxes,
that is the least-committal (i.e., the most imprecise) coherent
lower probability that extends $\lpr$ to a larger domain:

\begin{proposition}\label{prop:naturalextension}
Let $(\ldf_{\rm X},\udf_{\rm X})$ and $(\ldf_{\rm Y},\udf_{\rm Y})$ be two univariate
$p$-boxes. Let $\lpr$ be the lower probability defined on the set
$\mathcal{A}^{*}$ 
by means of Eqs.~\eqref{eq:def-marginal-lpr-x}
and~\eqref{eq:def-marginal-lpr-y}. The natural extension
$\underline{E}$ of $\lpr$ to ${\mathcal E}$ is given by
\[
\underline{E}(A_{(x,y)})=C_{\rm L}(\ldf_{\rm X}(x),\ldf_{\rm Y}(y)) \mbox{ and }
\underline{E}(A_{(x,y)}^c)=1-C_{\rm M}(\udf_{\rm X}(x),\udf_{\rm Y}(y)),
\]
for every $x,y\in\realsexp$. As a consequence, the bivariate $p$-box
$(\ldf,\udf)$ associated with $\underline{E}$ is given by:
\[
\ldf(x,y)=C_{\rm L}(\ldf_{\rm X}(x),\ldf_{\rm Y}(y)) \mbox{ and }\udf(x,y)=C_{\rm M}(\udf_{\rm X}(x),\udf_{\rm Y}(y)).
\]
\end{proposition}

\begin{proof}
The lower probability $\lpr$ is coherent from the previous lemma,
and in addition its associated credal set is
$\mathcal{M}(\lpr)=\mathcal{M}(C_{\rm L},C_{\rm M})$. The natural
extension of $\lpr$ to the set $\mathcal{E}$ is given by:
\[
\begin{array}{r c l}
\underline{E}(A_{(x,y)})&=&\inf_{P\in\mathcal{M}(\lpr)}F_{\rm P}(x,y)\\
                                               &=&\inf_{P\in\mathcal{M}(C_{\rm L},C_{\rm M})}F_{\rm P}(x,y)=C_{\rm L}(\ldf_{\rm X}(x),\ldf_{\rm Y}(y)).\\
\underline{E}(A_{(x,y)}^c)&=&\inf_{P\in\mathcal{M}(\lpr)}(1-P(A_{(x,y)}))=1-\sup_{P\in\mathcal{M}(\lpr)}F_{\rm P}(x,y)\\
                                                &=&1-\sup_{P\in\mathcal{M}(C_{\rm L},C_{\rm M})}F_{\rm P}(x,y)=1-C_{\rm M}(\udf_{\rm X}(x),\udf_{\rm Y}(y)).\\
\end{array}
\]
The second part is an immediate consequence of the first.
\end{proof}


The intuition of this result is clear: if we want to build the joint
$p$-box $(\ldf,\udf)$ from two given marginals $(\ldf_{\rm
X},\udf_{\rm X}),(\ldf_{\rm Y},\udf_{\rm Y})$, and we have no
information about the interaction between the underlying variables
$X$, $Y$, we should consider the largest, or most conservative,
imprecise copula: $(C_{\rm L},C_{\rm M})$. This corresponds to
combining the compatible univariate distribution functions by means
of all possible copulas, and then taking the envelopes of the
resulting set of bivariate distribution functions. What
Proposition~\ref{prop:naturalextension} shows is that this procedure
is equivalent to considering the \emph{natural extension} of the
associated coherent lower probabilities, and then take its
associated bivariate $p$-box. In
other words, the following diagram commutes: 
\begin{center}
\begin{picture}(300,100)
\put(35,80){$\lpr_{\rm X},\lpr_{\rm Y}$}
\put(250,80){$\underline{E}=\lpr_{(\ldf,\udf)}$}
\put(8,10){$(\ldf_{\rm X},\udf_{\rm X}),(\ldf_{\rm Y},\udf_{\rm
Y})$} \put(240,10){$(\ldf,\udf)$}
\put(260,45){Eq.~\eqref{eq:lpr-from-pbox}}
\put(-18,45){Eqs.~\eqref{eq:def-marginal-lpr-x},~\eqref{eq:def-marginal-lpr-y}}
\put(50,22){\vector(0,1){50}}
\put(253,22){\vector(0,1){50}} \put(75,82){\vector(1,0){165}}
\put(100,12){\vector(1,0){134}} \put(120,88){Natural extension}
\put(145,18){$(C_{\rm L},C_{\rm M})$}
\end{picture}
\end{center}

\subsection{Independent products of random variables}\label{sec:strong-product}

Next, we consider another case of interest: that where the variables
$X,Y$ are assumed to be independent. 
Under imprecise information, there is more than one way to model the
notion of independence; see \cite{couso2000a} for a survey on this
topic. Because of this, there is more than one manner in which we
can say that a coherent lower prevision $\lpr$ on the product space
is an \emph{independent product} of its marginals $\lpr_{\rm
X},\lpr_{\rm Y}$. Since the formalism considered in this paper can
be embedded into the theory of coherent lower previsions, here we
shall consider the notions of \emph{epistemic irrelevance} and
\emph{independence}, which seem to be more sound under the
behavioural interpretation that is at the core of this theory.

The study of independence under imprecision suffers from a number of
drawbacks when the underlying possibility spaces are infinite
\cite{miranda2014}. Because of this fact, we shall consider that the
variables $X,Y$ under study take values in respective finite spaces
$\pspace,\apspace$. Then the available information about these
variables is given by a coherent lower prevision $\lpr$ on
$\gambles(\pspace\times\apspace)$. We shall denote by $\lpr_{\rm
X},\lpr_{\rm Y}$ its respective marginals on
$\gambles(\pspace),\gambles(\apspace)$. Note that, similarly to
Eq.~\eqref{eq:pbox-from-lpr}, we can consider the bivariate $p$-box
$(\ldf,\udf)$ induced by $\lpr$ on $\pspace\times\apspace$, and also
the univariate $p$-boxes $(\ldf_{\rm X},\udf_{\rm X}),(\ldf_{\rm
Y},\udf_{\rm Y})$ induced by $\lpr_{\rm X},\lpr_{\rm Y}$ on
$\pspace,\apspace$.

We say then that the random variable $Y$ is {\em epistemically
irrelevant} to $X$ when
\begin{equation*}
 \lpr_{\rm X}(f|y):=\lpr_{\rm X}(f(\cdot,y)) \ \forall f\in\gambles(\pspace\times\apspace),y\in\apspace.
\end{equation*}
The variables $X,Y$ are said to be \emph{epistemically independent}
when each of them is epistemically irrelevant to the other:
\begin{equation}\label{eq:epistemic-independence}
 \lpr_{\rm X}(f|y):=\lpr_{\rm X}(f(\cdot,y)) \text{ and } \lpr_{\rm Y}(f|x):=\lpr_{\rm Y}(f(x,\cdot))
\end{equation}
for every $f\in\gambles(\pspace\times\apspace), x\in\pspace,
y\in\apspace$.

Here a \emph{conditional lower prevision} $\lpr(\cdot|\pspace)$ on
$\pspace\times\apspace$ is a collection of coherent lower previsions
$\{\lpr(\cdot|x):x\in\pspace\}$, so that $\lpr(\cdot|x)$ models the
available information about the outcome of $(X,Y)$ when we know that
$X$ takes the value $x$.\footnote{Strictly speaking,
$\lpr(\cdot|\pspace)$ refers to the lower prevision conditional on
the partition $\{\{x\}\times\apspace: x\in\pspace\}$ of
$\pspace\times\apspace$, and we use $\lpr(f|x)$ to denote
$\lpr(f|\{x\}\times\apspace)$. The reason for this is that Walley's
formalism defines lower previsions conditional on partitions of the
possibility space \cite[Chapter~6]{walley1991}.} Note that given
$f\in\gambles(\pspace\times\apspace)$, $\lpr(f|\pspace)$ is the
gamble on $\pspace\times\apspace$ that takes the value $\lpr(f|x)$
on the set $\{x\}\times\apspace$. Analogous comments can be made
with respect to $\lpr(\cdot|\apspace)$.

If we have a coherent lower prevision $\lpr$ and conditional lower
previsions $\lpr(\cdot|\pspace),\lpr(\cdot|\apspace)$, we should
check if the information they encompass is globally consistent. This
can be done by means of the notion of \emph{(joint) coherence} in
\cite[Def~7.1.4]{walley1991}, and from this we can establish the
following definition:

\begin{definition}
\label{def:independent_product} Let $\lpr$ be a coherent lower
prevision on $\gambles(\pspace\times\apspace)$ with marginals
$\lpr_{\rm X},\lpr_{\rm Y}$. We say that $\lpr$ is an \emph{independent product}
when it is coherent with the conditional lower previsions
$\lpr_{\rm X}(\cdot|\apspace),\lpr_{\rm Y}(\cdot|\pspace)$ derived from
$\lpr_{\rm X},\lpr_{\rm Y}$ by means of Eq.~\eqref{eq:epistemic-independence}.
\end{definition}


Given $\lpr_{\rm Y}$ and $\lpr_{\rm Y}$, one example of independent product is
the \emph{strong product}, given by
\begin{equation}\label{eq:strong-product}
 \lpr_{\rm X} \boxtimes \lpr_{\rm Y}:=\inf\{\pr_{\rm X} \times \pr_{\rm Y}: \pr_{\rm X} \geq
 \lpr_{\rm X}, \pr_{\rm Y} \geq \lpr_{\rm Y}\},
\end{equation}
where $\pr_{\rm X}\times\pr_{\rm Y}$ refers to the linear prevision uniquely
determined by\footnote{Recall that this is possible because we are
assuming that the possibility spaces $\pspace,\apspace$ are finite;
to see that the procedure above may not work with infinite spaces,
we refer to \cite{miranda2014}.} the finitely additive probability
such that $(\pr_{\rm X}\times\pr_{\rm Y})(x,y)=\pr_{\rm X}(x)\cdot\pr_{\rm Y}(y) \ \forall
x\in\pspace,y\in\apspace$. The strong product is the joint model
satisfying the notion of \emph{strong independence}. However, it is
not the only independent product, nor is it the smallest one. In
fact, the smallest independent product of the marginal coherent
lower previsions $\lpr_{\rm X},\lpr_{\rm Y}$ is called their \emph{independent
natural extension}, and it is given, for every gamble $f$ on
$\pspace\times\apspace$, by
\begin{multline*}
 (\lpr_{\rm X} \otimes \lpr_{\rm Y})(f)\\:=\sup\{\mu: f-\mu \geq
 g-\lpr_{\rm X}(g|\apspace)+h-\lpr_{\rm Y}(h|\pspace) \text{ for some } g,h\in
 \gambles(\pspace\times\apspace)\}.
\end{multline*}

One way of building independent products is by means of the
following condition:

\begin{definition}
A coherent lower prevision $\lpr$ on
$\gambles(\pspace\times\apspace)$ is called \emph{factorising} when
\begin{equation*}
 \lpr(f g)=\lpr(f\lpr(g)) \ \forall f\in\gambles^{+}(\pspace),
 g\in\gambles(\apspace)
\end{equation*}
and
\begin{equation*}
 \lpr(f g)=\lpr(g\lpr(f)) \ \forall f\in\gambles(\pspace),
 g\in\gambles^{+}(\apspace).
\end{equation*}
\end{definition}


Both the independent natural extension and the strong product are
factorising. Indeed, it can be proven \cite[Theorem~28]{cooman2011a}
that any factorising $\lpr$ is an independent product of its
marginals, but the converse is not true.
Under factorisation, the following result holds:

\begin{proposition}\label{prop:factorising} Let $(\ldf_{\rm X},\udf_{\rm X}),(\ldf_{\rm Y},\udf_{\rm Y})$ be
marginal $p$-boxes, and let $\lpr_{\rm X},\lpr_{\rm Y}$ be their associated
coherent lower previsions. Let $\lpr$ be a factorising coherent
lower prevision on $\gambles(\pspace\times\apspace)$ with these
marginals. Then it induces the bivariate $p$-box $(\ldf,\udf)$ given
by
\begin{equation*}
 \ldf(x,y)=\ldf_{\rm X}(x)\cdot\ldf_{\rm Y}(y)  \text{ and }
 \udf(x,y)=\udf_{\rm X}(x)\cdot\udf_{\rm Y}(y) \ \forall
(x,y)\in\pspace\times\apspace.
\end{equation*}
\end{proposition}

\begin{proof}
Let $x^*,y^*$ denote the maximum elements of $\pspace,\apspace$,
respectively. Since the indicator functions of
$A_{(x,y^*)},A_{(x^*,y)}$ are non-negative gambles such that
$A_{(x,y)}=A_{(x,y^*)}\cdot A_{(x^*,y)}$ and taking also into
account that $\lpr$ is factorising and positively homogeneous, we
get
\begin{equation*}
\lpr(A_{(x,y)})=\lpr(A_{(x,y^*)}\cdot
A_{(x^*,y)})=\lpr(A_{(x,y^*)})\cdot
\lpr(A_{(x^*,y)})=\ldf_{\rm X}(x)\cdot\ldf_{\rm Y}(y).
\end{equation*}
Similarly, if $\upr$ is the conjugate upper prevision of $\lpr$,
given by $\upr(f)=-\lpr(-f)$ for every
$f\in\gambles(\pspace\times\apspace)$, it holds that
\begin{align*}
\upr(A_{(x,y)})&=\upr(A_{(x,y^*)}\cdot
A_{(x^*,y)})\\&=-\lpr(A_{(x,y^*)} \cdot
(-A_{(x^*,y)}))=-\lpr(A_{(x,y^*)} \cdot (\lpr(-A_{(x^*,y)})))\\&=
-\lpr(-A_{(x,y^*)} \cdot (\upr(A_{(x^*,y)})))=
-\lpr(-A_{(x,y^*)})\cdot \upr(A_{(x^*,y)})\\&=
\upr(A_{(x,y^*)})\cdot \upr(A_{(x^*,y)})=\udf_{\rm
X}(x)\cdot\udf_{\rm Y}(y). \qedhere
\end{align*}
\end{proof}

From this it is easy to deduce that the $p$-box $(\ldf,\udf)$
induced by a factorising $\lpr$ is the envelope of the set of
bivariate distribution functions
\begin{equation*}
 \{F: F(x,y)=F_{\rm X}(x)\cdot F_{\rm Y}(y) \text{ for }
 F_{\rm X}\in(\ldf_{\rm X},\udf_{\rm X}),F_{\rm Y}\in(\ldf_{\rm Y},\udf_{\rm Y})\}.
\end{equation*}
In other words, the bivariate $p$-box can be obtained by applying
the imprecise version of Sklar's theorem
(Proposition~\ref{prop:Imp_Sklar2}) with the product copula.

Further, it has been showed in \cite{miranda2014} that a coherent
lower prevision $\lpr$ with marginals $\lpr_{\rm X},\lpr_{\rm Y}$ is factorising
if and only if it lies between the independent natural extension and
the strong product:
\begin{equation}\label{eq:factorising-bounds}
 \lpr_{\rm X}\otimes\lpr_{\rm Y} \leq \lpr \leq \lpr_{\rm X}\boxtimes\lpr_{\rm Y};
\end{equation}
as Walley showed in \cite[Section~9.3.4]{walley1991}, the
independent natural extension and the strong product do not coincide
in general, and this means that there may be an infinite number of
factorising coherent lower previsions with marginals
$\lpr_{\rm X},\lpr_{\rm Y}$. What Proposition~\ref{prop:factorising} tells us is
that all these factorising coherent lower previsions induce the same
bivariate $p$-box: the one determined by the product copula on the
marginal $p$-boxes.

Interestingly, this applies to other independence conditions that
guarantee the factorisation, such as the Kuznetsov property
\cite{cozman2014,cooman2011a}. This would mean that any Kuznetsov
product of the marginals $\lpr_{\rm X},\lpr_{\rm Y}$ induces the bivariate
$p$-box given by the product copula of the marginals.

However, not all independent products are factorising
\cite[Example~3]{cooman2011a}, and those that do not may
induce different $p$-boxes, as we show in the following example:

\begin{example}
Consider $\pspace=\apspace=\{0,1\}$. Let $\pr_1,\pr_2$ be the linear
previsions on $\gambles(\pspace)$ given by
\begin{equation*}
 \pr_1(f)=0.5 f(0)+ 0.5 f(1),\quad \pr_2(f)=f(0) \ \forall
 f\in\gambles(\pspace)
\end{equation*}
and let $\pr_3,\pr_4$ be the linear previsions on
$\gambles(\apspace)$ given by
\begin{equation*}
 \pr_3(f)=0.5 f(0)+ 0.5 f(1),\quad \pr_4(f)=f(0) \ \forall
 f\in\gambles(\apspace).
\end{equation*}
Consider the marginal lower previsions $\lpr_{\rm X}:=\min\{\pr_1,\pr_2\},
\lpr_{\rm Y}:=\min\{\pr_3,\pr_4\}$ on
$\gambles(\pspace),\gambles(\apspace)$, respectively. Applying
Eq.~\eqref{eq:strong-product}, their strong product is given by
\begin{multline*}
 \lpr_{\rm X}\boxtimes\lpr_{\rm Y}=\min\{\pr_1\times\pr_3, \pr_1\times\pr_4,
 \pr_2\times\pr_3, \pr_2\times\pr_4 \}\\ =
\min\{(0.25,0.25,0.25,0.25),(0.5,0,0.5,0),(0.5,0.5,0,0),(1,0,0,0)\},
\end{multline*}
where in the equation above a vector $(a,b,c,d)$ is used to denote
the vector of probabilities $\{(P(0,0),P(0,1),P(1,0),P(1,1))\}$.

Let $\lpr$ be the coherent lower prevision determined by the mass
functions
\begin{align*}
 \lpr:&=\min\{ \pr_1 \times (0.5 \pr_3+0.5\pr_4),
 (0.5\pr_1+0.5\pr_2)\times \pr_3, \pr_2\times\pr_4\}\\&=
 \min\{(0.375,0.125,0.375,0.125),(0.375,0.375,0.125,0.125),(1,0,0,0)\},
\end{align*}
where $(0.5 P_3+0.5 P_4)$ denotes the linear prevision on
$\gambles(\apspace)$ given by
\begin{equation*}
 (0.5 P_3+0.5 P_4)(f)=0.5 P_3(f)+0.5 P_4(f) \ \forall f\in\gambles(\apspace),
\end{equation*}
and similarly for $(0.5\pr_1+0.5\pr_2)$. Then the marginals of
$\lpr$ are also $\lpr_{\rm X},\lpr_{\rm Y}$. Moreover, since the extreme points
of $\solp(\lpr)$ are convex combinations of those of
$\solp(\lpr_{\rm X}\boxtimes\lpr_{\rm Y})$, we deduce that $\lpr$ dominates
$\lpr_{\rm X}\boxtimes\lpr_{\rm Y}$. Applying \cite[Proposition~5]{miranda2014},
we deduce that $\lpr$ is also an independent product of the marginal
coherent lower previsions $\lpr_{\rm X},\lpr_{\rm Y}$. Since it dominates
strictly the strong product, we deduce from
Eq.~\eqref{eq:factorising-bounds} that $\lpr$ is not factorising.

Now, since
\begin{equation*}
 \lpr(\{(0,0)\})=0.375 > 0.25
 =(\lpr_{\rm X}\boxtimes\lpr_{\rm Y})(\{(0,0)\}),
\end{equation*}
we see that the $p$-boxes associated with $\lpr$ and $\lpr_{\rm
X}\boxtimes\lpr_{\rm Y}$ differ. We conclude thus that not all
independent products induce the bivariate p-box that is the product
copula of its marginals. $\blacklozenge$
\end{example}

\begin{remark}
Interestingly, we can somehow distinguish between the strong product
and the independent natural extension in terms of bivariate
$p$-boxes, in the following way: if we consider the set of bivariate
distribution functions
\begin{equation*}
 \dfs:=\{\df_{\rm X} \times \df_{\rm Y}: \df_{\rm X} \in (\ldf_{\rm X},\udf_{\rm X}), \df_{\rm Y} \in(\ldf_{\rm Y},\udf_{\rm Y})\},
\end{equation*}
then it follows from Eq.~\eqref{eq:strong-product} that
\begin{equation}\label{eq:charac-sp}
 \lpr_{\rm X} \boxtimes \lpr_{\rm Y}:=\inf\{\pr: \df_{\rm P} \in\dfs\}.
\end{equation}
This differs from the coherent lower prevision given by
\begin{equation*}
 \lpr:=\min\{\pr: \df_{\rm P} \in (\ldf_{\rm X}\cdot \ldf_{\rm Y}, \udf_{\rm X}\cdot \udf_{\rm Y})\},
\end{equation*}
which will be in general more imprecise than the independent natural
extension $\lpr_{\rm X}\otimes\lpr_{\rm Y}$. Moreover, a
characterisation similar to Eq.~\eqref{eq:charac-sp} cannot be made
for the independent natural extension, in the sense that there is no
set of copulas ${\mathcal C}$ such that
\begin{multline*}
 \lpr_{\rm X} \otimes \lpr_{\rm Y}:=\inf\{\pr: \df_{\rm P}=C(\df_{\rm X},\df_{\rm Y}) \\ \text{ for some }C\in{\mathcal C},\df_{\rm X} \in (\ldf_{\rm X},\udf_{\rm X}), \df_{\rm Y} \in(\ldf_{\rm Y},\udf_{\rm Y})\};
\end{multline*}
indeed, just by considering the precise case we see that ${\mathcal
C}$ should consist just of the product copula, and this would give
back the definition of the strong product. $\blacklozenge$
\end{remark}

\section{Stochastic orders and copulas}\label{sec:stochastic-orders}

Next, we are going to apply the previous results to characterize the
preferences encoded by $p$-boxes. To this end, let us first of all
recall some basic notions on stochastic orders (see
\cite{levy1998,muller2002,shaked2006} for more information):

\begin{definition}\label{de:sto-dominance}
Given two univariate random variables $X$ and $Y$ 
with respective distribution functions $F_{\rm X}$ and $F_{\rm Y}$,
we say that $X$ {\em stochastically dominates} $Y$, and
denote it $X\fsd Y$, when
$F_{\rm X}(t)\leq F_{\rm Y}(t)$ for any $t$.
\end{definition}
This is one of the most extensively used methods for the comparison
of random variables. 
It is also called \emph{first order stochastic dominance}, so as to
distinguish it from the (weaker) notions of second, third, ..., n-th order stochastic dominance. 

An alternative for the comparison of random variables is statistical
preference. 
\begin{definition}[\cite{schuymer2003b,schuymer2003}]
Given two univariate random variables $X$ and $Y$, $X$ is said to be
{\em statistically preferred} to $Y$ if $P(X \geq Y)\geq P(Y \geq
X)$. This is denoted by $X\sp Y$.
\end{definition}
This notion is particularly interesting when the variables $X,Y$ take
values in a qualitative scale \cite{Dubois03}.

In addition to comparing pairs of random variables, or, more
generally, couples of `elements', with a preorder relation, we may
be interested in comparing pairs of sets (of random variables or
other `elements') by means of the given order relation. We can
consider several different possibilities:
\begin{definition}\label{de:imprecise-orders_general}
Let $\succeq$ be a preorder over a set $S$. Given  $A, B\subseteq
S$, we say that:
\begin{enumerate}
\item $A\succeq_1 B$ if and only if for every $a\in A$, $b\in B$ it holds that $a\succeq b$.
\item $A\succeq_2 B$ if and only if there exists some $a\in A$ such that $a\succeq b$ for every $b\in B$.
\item $A\succeq_3 B$ if and only if for every $b\in B$ there is some $a\in A$ such that $a\succeq b$.
\item $A\succeq_4 B$ if and only if there are $a\in A$, $b\in B$ such that $a\succeq b$.
\item $A\succeq_5 B$ if and only if there is some $b\in B$ such that $a\succeq b$ for every $a\in A$.
\item $A\succeq_6 B$ if and only if for every $a\in A$ there is $b\in B$ such that $a\succeq b$.
\end{enumerate}
\end{definition}

The relations $\succeq_{i}$ in Definition
\ref{de:imprecise-orders_general} have been discussed in
\cite{montes2014} in the case that $\succeq$ is the stochastic
dominance relation $\fsd$ and in \cite{montes2014b} in the case of
statistical preference, showing that several of them are related to
decision criteria explored in the literature of imprecise
probabilities.

Figure~\ref{fig} illustrates some of these extensions. In
Figure~\ref{fig:FSD1}, $A\succeq_1 B$ because all the alternatives
in $A$ are better than all the alternatives in $B$; in
Figure~\ref{fig:FSD2}, $A\succeq_2 B$ because there is an optimal
element in $A$, $a_1$, that is preferred to all the alternatives in
$B$; Figure~\ref{fig:FSD4} shows an example of $A\succeq_4 B$
because there are alternatives $a_1\in A$ and $b_2\in B$ such that
$a_1\succeq b_2$; finally, Figure~\ref{fig:FSD5} shows an example of
$A\succeq_5 B$ because there is a worst element in $B$, $b_1$ that
is dominated by all the elements in $A$. The difference between the
second and the third extensions (resp., fifth and sixth) lies in the
existence of a maximum (resp., minimum) or a supremum (resp.,
infimum) element in $A$ (resp., $B$).
\begin{figure}
\centering
        \begin{subfigure}[b]{0.3\textwidth}
                \includegraphics[width=\textwidth]{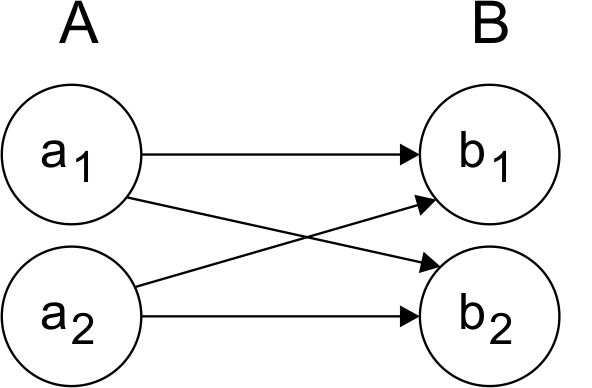}
                \caption{$A\succeq_1 B$}
                \label{fig:FSD1}
        \end{subfigure}\qquad%
        ~ 
        \begin{subfigure}[b]{0.3\textwidth}
                \includegraphics[width=\textwidth]{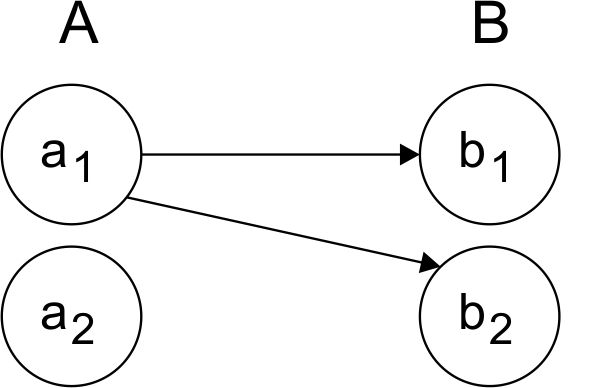}
                \caption{$A\succeq_2 B$}
                \label{fig:FSD2}
        \end{subfigure}\qquad
        ~ 
        \begin{subfigure}[b]{0.3\textwidth}
                \includegraphics[width=\textwidth]{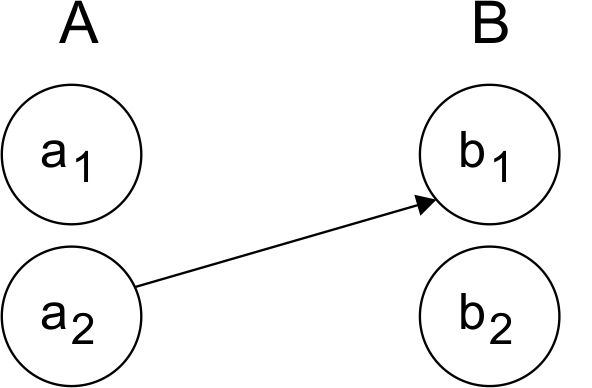}
                \caption{$A\succeq_4 B$}
                \label{fig:FSD4}
        \end{subfigure}\qquad\quad
                \begin{subfigure}[b]{0.3\textwidth}
                \includegraphics[width=\textwidth]{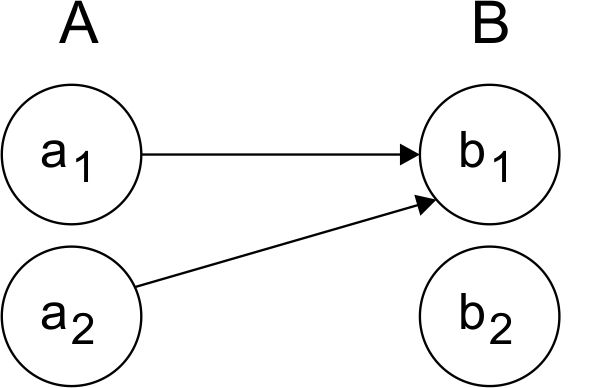}
                \caption{$A\succeq_5 B$}
                \label{fig:FSD5}
        \end{subfigure}\qquad
        \caption{Examples of the extensions of $\succeq_i$. In this picture $a_i\rightarrow b_j$ means $a_i\succeq b_j$.}\label{fig}
\end{figure}

\subsection{Univariate orders}
Although stochastic dominance does not imply statistical preference
in general\footnote{Consider for instance the case where the joint
distribution is given by $P(X=0,Y=0.5)=0.2,
P(X=0.5,Y=0)=P(X=1,Y=0)=P(X=0.5,Y=1)=0.1$ and $P(X=1,Y=1)=0.5$. Then
$X$ and $Y$ are equivalent with respect to stochastic dominance
because their cumulative distribution functions coincide; however,
$Y$ is strictly statistically preferred to $X$.
}, in the
univariate case a number of sufficient conditions have been
established for the implication, in terms of the copula that
determines the joint distribution from the marginal ones. This is
for instance the case when:
\begin{description}
\item[(SD-SP1)] $X$ and $Y$ are stochastically independent random variables, i.e., they are linked by the product copula (see \cite{schuymer2003b,schuymer2005,montes2011b});
\item[(SD-SP2)] $X$ and $Y$ are absolutely continuous random variables and they are coupled by an Archimedean copula (see \cite{Montes2010}).
\item[(SD-SP3)] $X,Y$ are either comonotonic or countercomonotonic, and they are both either simple or absolutely continuous.
\end{description}
In such cases, the implication transfers to the relations comparing
sets of random variables, by means of the following lemma. Its proof
is immediate and therefore omitted.

\begin{lemma}
\label{le:set_order} Let $\succeq$ be a preorder in a set $S$ and
$A, B\subseteq S$. Let also $\sqsupseteq$ be a preorder that extends
$\succeq$, i.e. $x\succeq y \Rightarrow x\sqsupseteq y \ \forall
x,y\in S$. Then, $A\succeq_{i} B\Rightarrow A\sqsupseteq_{i} B$ for
all $i=1,\ldots,6$.
\end{lemma}

Here $A$, $B$ are sets of random variables, denoted ${\mathcal V}_{\rm X}$, ${\mathcal V}_{\rm Y}$.
The following special case of Lemma \ref{le:set_order} is an instance.
\begin{proposition}
Consider two sets of random variables ${\mathcal V}_{\rm X},
{\mathcal V}_{\rm Y}$. Assume that any $X\in{\mathcal V}_{\rm X},
Y\in{\mathcal V}_{\rm Y}$ satisfy one of the conditions
(SD-SP1)$\div$(SD-SP3)
above. 
Then, for all $i=1,\ldots,6$:
\[
{\mathcal V}_{\rm X}\succeq_{\rm SD_i} {\mathcal V}_{\rm Y}
\Rightarrow {\mathcal V}_{\rm X}\succeq_{\rm SP_i} {\mathcal V}_{\rm
Y}.
\]
\end{proposition}
\begin{proof}
As we have remarked, conditions (SD-SP1)$\div$(SD-SP3) above ensure
that the statistical preference relation is an extension of
stochastic dominance. The result follows from Lemma
\ref{le:set_order}.
\end{proof}

\subsection{Bivariate orders}

Next we consider the following extension of stochastic dominance to
the bivariate case:


\begin{definition}\label{de:bivariate-order}
Let $X=(X_1,X_2)$ and $Y=(Y_1,Y_2)$ be two random vectors with
respective bivariate distribution functions $F_{\rm X_1,X_2}$ and
$F_{\rm Y_1,Y_2}$. We say that $(X_1,X_2)$ {\em stochastically
dominates} $(Y_1,Y_2)$, and denote it $(X_1,X_2)\succeq_{\rm
SD}(Y_1,Y_2)$, if $F_{\rm X_1,X_2}(s,t)\leq F_{\rm Y_1,Y_2}(s,t)$
for all $(s,t)\in\mathbb{R}^{2}$.
\end{definition}

This definition establishes a way of comparing two bivariate vectors
$X=(X_1,X_2)$, $Y=(Y_1,Y_2)$ in case their associated distribution
functions are precisely known. However, it is not uncommon to have
uncertain information about these distribution functions, that we
can model by means of respective sets of distribution functions
$\dfs_{\rm X},\dfs_{\rm Y}$. If we now take
Definition~\ref{de:imprecise-orders_general} into account, we can
propose a generalisation of Definition~\ref{de:bivariate-order} to
the imprecise case:

\begin{definition}
Let $X=(X_1,X_2)$ and $Y=(Y_1,Y_2)$ be two random vectors with
respective sets of bivariate distribution functions $\dfs_{\rm X},\dfs_{\rm Y}$.
We say that $(X_1,X_2)$ {\em i-stochastically dominates}
$(Y_1,Y_2)$, and denote it $(X_1,X_2)\succeq_{\rm SD_i}(Y_1,Y_2)$,
if $\dfs_{\rm X} \leq_{\rm i} \dfs_{\rm Y}$.
\end{definition}

Since by Remark \ref{rem:copula} copulas can be interpreted as
bivariate distribution functions, the extensions $\leq_{\rm i}$ are
also applicable to them.

Note that the sets of distribution functions $\dfs_{\rm X},\dfs_{\rm Y}$ may be
obtained by combining two respective marginal $p$-boxes by means of
a set of copulas. In that case, we may study to which extent the
relationships between the sets $\dfs_{\rm X},\dfs_{\rm Y}$ can be determined by
means of the relationships between their marginal univariate
$p$-boxes. In other words, if we have information stating that $X_1$
stochastically dominates $Y_1$ and $X_2$ stochastically dominates
$Y_2$, we may wonder in which cases the pair $(X_1,X_2)$
$i$-stochastically dominates $(Y_1,Y_2)$. The following result gives
an answer to this question:

\begin{proposition}\label{prop:lo}
Given two random vectors $X=(X_1,X_2)$, $Y=(Y_1,Y_2)$, let
$(\ldf_{\rm X_1},\udf_{\rm X_1}), (\ldf_{\rm X_2},\udf_{\rm X_2}),
(\ldf_{\rm Y_1},\udf_{\rm Y_1})$, $(\ldf_{\rm Y_2},\udf_{\rm Y_2})$
be the marginal $p$-boxes associated with $X_1$, $X_2$, $Y_1$, $Y_2$
respectively. Let $\C_{\rm X}$ and $\C_{\rm Y}$ be two sets of
copulas. Define the following sets of bivariate distribution
functions $\dfs_{\rm X},\dfs_{\rm Y}$:
\begin{align*}
  \dfs_{\rm X}&:=\{C(F_{\rm X_1},F_{\rm X_2}): C\in\C_{\rm X}, F_{\rm X_1}\in
  (\ldf_{\rm X_1},\udf_{\rm X_1}), F_{\rm X_2}\in
  (\ldf_{\rm X_2},\udf_{\rm X_2})\}, \\
  \dfs_{\rm Y}&:=\{C(F_{\rm Y_1},F_{\rm Y_2}): C\in\C_{\rm Y}, F_{\rm Y_1}\in
  (\ldf_{\rm Y_1},\udf_{\rm Y_1}), F_{\rm Y_2}\in
  (\ldf_{\rm Y_2},\udf_{\rm Y_2})\}.
\end{align*}

Consider $i\in\{1,\dots,6\}$ and assume that $(\ldf_{\rm
X_j},\udf_{\rm X_j})\leq_{\rm i}(\ldf_{\rm Y_j},\udf_{\rm Y_j})$ for
$j=1,2$. Then:
\[
\C_{\rm X}\leq_{\rm i}\C_{\rm Y} \Rightarrow (X_1,X_2)\succeq_{\rm
SD_i} (Y_1,Y_2).
\]
\end{proposition}
\begin{proof}
\begin{itemize}\setlength{\itemindent}{0.14in}
\item[$(i=1)$] We know that:
\[
\begin{array}{l}
\forall F_{\rm X_j}\in (\ldf_{\rm X_j},\udf_{\rm X_j}), F_{\rm Y_j}\in (\ldf_{\rm Y_j},\udf_{\rm Y_j}), F_{\rm X_j}\leq F_{\rm Y_j}, (j=1,2);\\
\forall C_{\rm X}\in\C_{\rm X},C_{\rm Y}\in \C_{\rm Y}, C_{\rm
X}\leq C_{\rm Y},\\
\end{array}
\]
Consider $F_{\rm X}\in\dfs_X$ and $F_{\rm Y}\in \dfs_Y$. They can be expressed in the
following way: $F_{\rm X}(x,y)=C_{\rm X}(F_{\rm X_1}(x),F_{\rm
X_2}(y))$ and $F_{\rm Y}(x,y)=C_{\rm Y}(F_{\rm Y_1}(x),F_{\rm
Y_2}(y))$, where $C_{\rm X}\leq C_{\rm Y}$. Then:
\[
\begin{array}{l c l}
F_{\rm X}(x,y)&=&C_{\rm X}(F_{\rm X_1}(x),F_{\rm X_2}(y))\leq C_{\rm X}(F_{\rm Y_1}(x),F_{\rm Y_2}(y))\\
         &\leq& C_{\rm Y}(F_{\rm Y_1}(x),F_{\rm Y_2}(y))=F_{\rm Y}(x,y),
\end{array}
\]
where the inequalities hold because copulas are component-wise increasing.
\item[$(i=2)$] We know that:
\[
\begin{array}{l}
\exists F_{\rm X_j}^*\in (\ldf_{\rm X_j},\udf_{\rm X_j}) \mbox{ s.t.
}F_{\rm X_j}^*\leq F_{\rm Y_j} \quad \forall F_{\rm Y_j}\in
(\ldf_{\rm Y_j},\udf_{\rm Y_j}), (j=1,2).\\
\exists C_{\rm X}^*\in\C_{\rm X} \mbox{ s.t. }C_{\rm
X}^*\leq C_{\rm Y}\quad \forall C_{\rm Y}\in\C_{\rm Y}.
\end{array}
\]
Consider $F_{\rm X}(x,y):=C_{\rm X}^*(F_{\rm X_1}^*(x),F_{\rm
X_2}^*(y))$, and let us see that $F_{\rm X}\leq F_{\rm Y}$ for any
$F_{\rm Y}=C_{\rm Y}(F_{\rm Y_1},F_{\rm Y_2})$ in $ \dfs_Y$:
\[
\begin{array}{lcl}
F_{\rm X}(x,y)&=&C_{\rm X}^*(F_{\rm X_1}^*(x),F_{\rm X_2}^*(y))\leq
C_{\rm X}^*(F_{\rm Y_1}(x),F_{\rm X_2}(y))\\
&\leq & C_{\rm
Y}(F_{\rm Y_1}(x),F_{\rm X_2}(y))=F_{\rm Y}(x,y).
\end{array}
\]
\item[($i=3$)] We know that:
\[
\begin{array}{l}
\forall F_{\rm Y_j}\in (\ldf_{\rm Y_j},\udf_{\rm Y_j}), \exists
F^*_{\rm X_j}\in(\ldf_{\rm X_j},\udf_{\rm X_j}) \mbox{ s.t. }F^*_{\rm X_j}\leq F_{\rm Y_j}, (j=1,2).\\
\forall C_{\rm Y}\in\C_{\rm Y} \ \exists C^*_{\rm X}\in\C_{\rm X}
\mbox{ s.t. }C^*_{\rm X}\leq C_{\rm Y}.
\end{array}
\]
Let $F_{\rm Y}\in\dfs_Y$. Then, there are $C_{\rm
Y}\in\mathcal{C}_{Y}, F_{\rm Y_1}\in(\ldf_{\rm Y_1},\udf_{\rm Y_1})$
and $F_{\rm Y_2}\in(\ldf_{\rm Y_2},\udf_{\rm Y_2})$ such that
$F_{\rm Y}(x,y)=C_{\rm Y}(F_{\rm Y_1}(x),F_{\rm Y_2}(y))$. Let us
check that there is $F_{\rm X}$ in $\dfs_X$ such that $F_{\rm X}\leq
F_{\rm Y}$. Let $F_{\rm X}(x,y)=C^*_{\rm X}(F^*_{\rm
X_1}(x),F^*_{\rm X_2}(y))$. Then:
\[
\begin{array}{lcl}
F_{\rm X}(x,y)&=&C^*_{\rm X}(F^*_{\rm X_1}(x),F^*_{\rm X_2}(y))\leq
C^*_{\rm X}(F_{\rm Y_1}(x),F_{\rm Y_2}(y))\\ &\leq&
C_{\rm Y}(F_{\rm Y_1}(x),F_{\rm Y_2}(y))=F_{\rm Y}(x,y).
\end{array}
\]
\item[($i=4$)] We know that:
\[
\begin{array}{l}
\exists
F_{\rm X_j}^*\in(\ldf_{\rm X_j},\udf_{\rm X_j}),F_{\rm Y_j}^*\in(\ldf_{\rm Y_j},\udf_{\rm Y_j})
\mbox{ s.t. }F_{\rm X_j}^*\leq F_{\rm Y_j}^*, (j=1,2).\\
\exists C_{\rm X}^*\in\C_{\rm X},C_{\rm Y}^*\in\C_{\rm Y} \mbox{ s.t. }C^*_{\rm X}\leq C^*_{\rm Y}.
\end{array}
\]
Let us consider the distribution functions $F_{\rm X}(x,y)=C_{\rm X}^*(F_{\rm X_1}^*(x),F_{\rm X_2}^*(y))$ and $F_{\rm Y}(x,y)=C_{\rm Y}^*(F_{\rm Y_1}^*(x),F_{\rm Y_2}^*(y))$. It holds that $F_{\rm X}\leq
F_{\rm Y}$:
\[
\begin{array}{lcl}
F_{\rm X}(x,y)&=&C_{\rm X}^*(F_{\rm X_1}^*(x),F_{\rm X_2}^*(y))\leq
C_{\rm X}^*(F_{\rm Y_1}^*(x),F_{\rm Y_2}^*(y))\\
&\leq &C_{\rm Y}^*(F_{\rm Y_1}^*(x),F_{\rm Y_2}^*(y))=F_{\rm Y}(x,y).
\end{array}
\]
\end{itemize}
\begin{itemize}\setlength{\itemindent}{0.5in}
\item[($i=5,i=6$)] The proof of these two cases is analogous to that of $i=2$ and $i=3$, respectively. $\qedhere$
\end{itemize}
\end{proof}

\subsection{Natural extension and independent products}

To conclude this section, we consider the particular cases discussed in Sections~\ref{sec:natex}
and~\ref{sec:strong-product}: those where the bivariate $p$-box is
the natural extension or a factorising product.

By Proposition~\ref{prop:naturalextension}, the natural extension of
two marginal $p$-boxes $(\ldf_{\rm X},\udf_{\rm X})$ and $(\ldf_{\rm Y},\udf_{\rm Y})$ is
given by:
\begin{equation}\label{eq:naturalextension}
\ldf(x,y)=C_{\rm L}(\ldf_{\rm X}(x),\ldf_{\rm Y}(y)) \mbox{ and }\udf(x,y)=C_{\rm M}(\udf_{\rm X}(x),\udf_{\rm Y}(y)).
\end{equation}
This allows us to prove the following result:

\begin{corollary}
Consider marginal $p$-boxes $(\ldf_{\rm X_1},\udf_{\rm
X_1}),(\ldf_{\rm X_2},\udf_{\rm X_2}),(\ldf_{\rm Y_1},\udf_{\rm
Y_1})$ and $(\ldf_{\rm Y_2},\udf_{\rm Y_2})$. Let $(\ldf_{\rm
X},\udf_{\rm X})$ (resp., $(\ldf_{\rm Y},\udf_{\rm Y})$) be the
natural extension of the $p$-boxes $(\ldf_{\rm X_1},\udf_{\rm
X_1}),(\ldf_{\rm X_2},\udf_{\rm X_2})$ (resp., $(\ldf_{\rm
Y_1},\udf_{\rm Y_1}),(\ldf_{\rm Y_2},\udf_{\rm Y_2})$) by means of
Eq.~\eqref{eq:naturalextension}. Then for $i=2,\dots,6$,
\[
(\ldf_{\rm X_j},\udf_{\rm X_j})\leq_{\rm i} (\ldf_{\rm
Y_j},\udf_{\rm Y_j}), j=1,2 \Rightarrow (X_1,X_2)\succeq_{\rm SD_i}
(Y_1,Y_2).
\]
\end{corollary}

\begin{proof}
Take $\mathcal{C}_{\rm X}=\mathcal{C}_{\rm Y}=\{C_{\rm L},C_{\rm
M}\}$ in Proposition~\ref{prop:lo}. Since $C_{\rm L}\leq C_{\rm L}$,
$C_{\rm M}\leq C_{\rm M}$ and $C_{\rm L}\leq C_{\rm M}$, we get
$\mathcal{C}_{\rm X}\leq_{\rm i}\mathcal{C}_{\rm Y}$,
$i=2,\ldots,6$. Then, Proposition~\ref{prop:lo} ensures $\{\ldf_{\rm
X},\udf_{\rm X}\}\leq_{\rm i}\{\ldf_{\rm Y},\udf_{\rm Y}\} \
(i=2,\ldots,6)$. It is not difficult to check then that this implies
also $(\ldf_{\rm X},\udf_{\rm X})\leq_{\rm i}(\ldf_{\rm Y},\udf_{\rm
Y}) \ (i=2,\ldots,6)$, because of the special form of
$\mathcal{C}_{\rm X}$, $\mathcal{C}_{\rm Y}$.
\end{proof}

To see that the result does not hold for $\leq_1$, consider the
following example:

\begin{example}
For $j=1,2$, let $\ldf_{\rm X_j}=\udf_{\rm X_j}=\ldf_{\rm Y_j}=\udf_{\rm Y_j}$ be
the distribution function associated with the uniform probability
distribution on $[0,1]$, given by $F(x)=x$ for every $x\in[0,1]$.
Then trivially
\begin{equation*}
F=(\ldf_{\rm X_j},\udf_{\rm X_j})\leq_{\rm 1} (\ldf_{\rm
Y_j},\udf_{\rm Y_j})=F \ \forall j=1,2.
\end{equation*}
However, $(\ldf_{\rm X},\udf_{\rm X})\nleq_{\rm 1}(\ldf_{\rm
Y},\udf_{\rm Y})$, since $C_{\rm M}(F,F)\in (\ldf_{\rm X},\udf_{\rm
X}), C_{\rm L}(F,F)\in (\ldf_{\rm Y},\udf_{\rm Y})$ and
\begin{align*}
C_{\rm M}(F,F)(0.5,0.5)&=C_{\rm M}(F(0.5),F(0.5))=C_{\rm
M}(0.5,0.5)=0.5 > 0
\\&=C_{\rm L}(0.5,0.5)=C_{\rm L}(F(0.5),F(0.5))=C_{\rm L}(F,F)(0.5,0.5). \blacklozenge
\end{align*}
\end{example}

On the other hand, Proposition~\ref{prop:factorising} implies that,
given two finite spaces $\pspace,\apspace$, any factorising coherent
lower prevision $\lpr$ on $\gambles(\pspace\times\apspace)$
determines a bivariate $p$-box that is the product of its marginal
$p$-boxes by means of the product copula. Taking this property into
account, we can compare two factorising independent products in
terms of the relationships between their marginals. From
Proposition~\ref{prop:lo}, we deduce the following:

\begin{corollary}
Consider marginal $p$-boxes
$(\ldf_{\rm X_1},\udf_{\rm X_1}),(\ldf_{\rm Y_1},\udf_{\rm Y_1})$,
$(\ldf_{\rm X_2},\udf_{\rm X_2})$ and $(\ldf_{\rm Y_2},\udf_{\rm
Y_2})$,
and let us define the following sets of bivariate distribution
functions $\dfs_X,\dfs_Y$ by 
\begin{align*}
  \dfs_X&:=\{F_{\rm X_1} \cdot F_{\rm X_2}: F_{\rm X_1}\in
  (\ldf_{\rm X_1},\udf_{\rm X_1}), F_{\rm X_2}\in
  (\ldf_{\rm X_2},\udf_{\rm X_2})\}, \\
  \dfs_Y&:=\{F_{\rm Y_1} \cdot F_{\rm Y_2}: F_{\rm Y_1}\in
  (\ldf_{\rm Y_1},\udf_{\rm Y_1}), F_{\rm Y_2}\in
  (\ldf_{\rm Y_2},\udf_{\rm Y_2})\}.
\end{align*}
Then, for $i=1,\dots,6$,
\[
(\ldf_{\rm X_j},\udf_{\rm X_j})\leq_{\rm i} (\ldf_{\rm
Y_j},\udf_{\rm Y_j}), j=1,2 \Rightarrow (X_1,X_2)\succeq_{\rm SD_i}
(Y_1,Y_2). 
\]
\end{corollary}

\begin{proof}
The result is the particular case of Proposition~\ref{prop:lo} where
$\mathcal{C}_X=\mathcal{C}_Y=\{C_{\rm P}\}$.
\end{proof}

%
%
%

\section{Conclusions and open problems}\label{sec:conclusions}
In this work we have studied the extension of Sklar's theorem to an
imprecise framework, where instead of random variables precisely
described by their distribution functions, we have considered the
case when they are imprecisely described by p-boxes. For this aim,
we have introduced the notion of imprecise copula, and have proven
that if we link two marginal p-boxes by means of a set of copulas we
obtain a bivariate p-box whose associated lower probability is
coherent. Unfortunately, the main implication of Sklar's theorem
does not hold in the imprecise framework: there exist coherent
bivariate p-boxes that are not uniquely determined by their
marginals.

We have investigated two particular cases: on the one hand, we
considered the absence of information about the copula that links
the marginals. In that case, we end up with the natural extension of
the marginal p-boxes, that can be expressed in terms of the
{\L}ukasiewicz and the minimum copulas. On the other hand, we looked
upon the case where the marginal distributions satisfy the condition
of epistemic independence, and showed that the joint p-box can be
obtained in most, but not all cases, by means of the product copula.


There are a few open problems that arise from our work in this
paper: on the one hand, we should deepen the study of the properties
of imprecise copulas from the point of view of aggregation
operators. With respect to Sklar's theorem, we intend to look for
sufficient conditions for a bivariate p-box to be determined as an
imprecise copula of its marginals. A third open problem would be the
study in the imprecise case of the other extensions of stochastic
dominance to the bivariate case, based on the comparisons of
survival functions or expectations. Finally, it would be interesting
to generalize our results to the $n$-variate case. An interesting
work in this respect was carried out by Durante and Spizzichino in
\cite{Durante}.

\section*{Acknowledgements}

The research in this paper started during a stay of Ignacio Montes
at the University of Trieste supported by the Campus of
International Excellence of the University of Oviedo. Ignacio Montes
and Enrique Miranda also acknowledge the financial support by the
FPU grant AP2009-1034 and by project MTM2010-17844. Renato Pelessoni
and Paolo Vicig acknowledge partial support by the FRA2013 grant
`Models for Risk Evaluation, Uncertainty Measurement and Non-Life
Insurance Applications'. Finally, we would like to thank Jasper de
Bock and the anonymous reviewers for their useful suggestions.

\bibliographystyle{plain}

\end{document}